\documentclass[10pt,draft,reqno]{amsart}
  
\usepackage{amsfonts}
\usepackage{amsmath}
\usepackage{amsthm}
\usepackage{amssymb}
\usepackage{lastpage}
\usepackage{fancyhdr}
\usepackage{fancybox}
\usepackage{mathrsfs}

     \makeatletter
     \def\section{\@startsection{section}{1}
     \z@{.7\linespacing\@plus\linespacing}{.5\linespacing}
     {\bfseries
     \centering
     }}
     \def\@secnumfont{\bfseries}
     \makeatother
\def\phi{\varphi }
\def\epsilon{\varepsilon}

\def\ch{\cosh}
\newtheorem{theorem}{Theorem}[section]
\newtheorem{lemma}[theorem]{Lemma}
\newtheorem{proposition}[theorem]{Proposition}
\newtheorem{corollary}[theorem]{Corollary}
\theoremstyle{definition}
\newtheorem{definition}[theorem]{Definition}
\newtheorem{example}[theorem]{Example}
\newtheorem{examples}[theorem]{Examples}
\newtheorem{remark}[theorem]{Remark}
\newtheorem{remarks}[theorem]{Remarks}
\newtheorem{facts}[theorem]{Facts}
\numberwithin{equation}{section}
\begin{document}

\title[Association schemes and hypergroups]{Generalized Commutative 
Association Schemes, Hypergroups, and Positive Product Formulas}

\author{Michael Voit}
\address{Michael Voit: Fakult\"at Mathematik,
 Technische Universit\"at Dortmund, 
          Vogel\-poths\-weg 87,
          D-44221 Dortmund, Germany}
\email{michael.voit@math.tu-dortmund.de}



\begin{abstract}
It is well known that finite commutative association schemes in the sense of
the monograph of Bannai and Ito lead to finite commutative hypergroups with 
positive dual convolutions and even dual hypergroup structures. 
In this paper we present several discrete generalizations of
association schemes which also lead to associated hypergroups.
We show that  discrete commutative hypergroups
associated with such  generalized association schemes admit dual positive 
 convolutions at least on the support of the Plancherel measure.
We hope that examples for this theory will lead to the existence of new 
dual positive product formulas in near future.
\end{abstract}

\maketitle

\par\bigskip\noindent
 This paper is devoted to Herbert Heyer on the occasion of
his eightieth birthday.

\section{Motivation}

The following setting appears
quite often in the theory of Gelfand pairs, spherical functions, and
associated  special functions:

 Let $(G_n,H_n)_{n\in\mathbb N}$ be a sequence of
Gelfand pairs , i.e., locally compact groups $G_n$ with compact subgroups
$H_n$ such that the Banach algebras $M_b(G_n||H_n)$ of all bounded,
 signed $H_n$-biinvariant Borel measures on $G_n$ are commutative.
 Assume also that the 
double coset spaces $G_n//H_n:=\{H_ngH_n: \> g\in G_n\}$ with the
quotient topology  are homeomorphic with
some fixed locally compact space $D$. Then the  space $M_b(D)$ 
 carries the  canonical double coset hypergroup convolutions $*_n$.

A non-trivial  $H_n$-biinvariant continuous function $\phi_n\in C(G_n)$
is called spherical if the product formula
\begin{equation}\label{prod-formula-allg}
\phi_n(g)\phi_n(h)=\int_{H_n} \phi_n(gkh)\> d\omega_{H_n}(k) 
\quad(g,h\in G_n)
\end{equation}
holds with the  normalized Haar measure
 $\omega_{H_n}$ of $H_n$. Via $G_n//H_n\equiv D$, we may identify the
 spherical functions of $(G_n,H_n)$ with  the nontrivial continuous
 functions on  $D$, which are multiplicative
w.r.t.~$*_n$.

For
all relevant examples of such series $(G_n,H_n)_n$, the spherical functions  are
parameterized by some  spectral parameter set $\chi(D)$ independent on $n$, and
 the associated functions $\phi_n:\chi(D)\times D\to\mathbb C$
 can be embedded into a 
family of special functions which depend
analytically on  $n$  
in some domain $A\subset\mathbb C$, where these functions 
are
 spherical 
for some integers $n$.
 In many cases, we can determine these
 special functions  and obtain  concrete versions of the product
formula (\ref{prod-formula-allg}), in which $n$ appears as a parameter.
 Based on
Carleson's theorem, a principle of analytic continuation 
(see e.g.~\cite{Ti}, p.186), 
it is often easy to extend the positive product formula for  $\phi_n$
in the group cases to a continuous range 
of parameters. Usually, this extension leads to a continuous family of
commutative hypergroups.

Classical examples of such positive product formulas are  the
well-known  product formulas of Gegenbauer for the normalized ultraspherical 
polynomials
\begin{equation}\label{classical-ultra-pol}
R_k^{(\alpha,\alpha)}(x)= \>_2F_1(2\alpha+k+1, -k, \alpha+1; (1-x)/2)
\quad(k\in\mathbb N_0)
\end{equation}
on $D=[-1,1]$ with $\chi(D)= \mathbb N_0$ and
 for the modified Bessel functions
\begin{equation}\label{classical-Bessel-funct}
\Lambda_\alpha(x,y):=j_\alpha(xy)
\quad\text{with}\quad j_\alpha(z):=\>_0F_1(\alpha+1; -z^2/4)
\quad(y\in\mathbb C)
\end{equation}
on $D=[0,\infty[$ with $\chi(D)=\mathbb C$; 
see e.g.~the survey \cite{A}. In both cases, this works for
$\alpha\in[-1/2,\infty[$ where $\alpha=(n-1)/2$ corresponds
to the  Gelfand
    pair $(G_n,H_n)$ with
$G_n=SO(n+2), \> H_n=SO(n+1)$ and 
$G_n=SO(n+1)\ltimes\mathbb R^{n+1}, \> H_n=SO(n+1)$ respectively.
The continuous ultraspherical product formula  can be extended to
Jacobi polynomials \cite{G2} which generalizes 
the  product formulas for the
spherical functions of the projective spaces over
$\mathbb F=\mathbb R, \mathbb C$ and the quaternions $\mathbb H$.
Further prominent semisimple, rank one examples 
are the Gelfand pairs associated with
the hyperbolic spaces over
$\mathbb F=\mathbb R, \mathbb C,\mathbb H$
 with the groups
\begin{align}
\mathbb F=\mathbb R:& \quad\quad\quad G=SO_o(1,k), \quad K=SO(k)    \notag\\
\mathbb F=\mathbb C:& \quad\quad\quad G=SU(1,k), \quad K=S(U(1)\times U(k))
    \notag\\
\mathbb F=\mathbb H:& \quad\quad\quad G=Sp(1,k), \quad K=Sp(1)\times Sp(k) .
   \notag
\end{align}
Here, $D=[0,\infty[$ with $\chi(D)=\mathbb C$ where the
    spherical functions are Jacobi functions \cite{Ko}.
Besides these  classical rank-one examples, reductive examples as well as 
 several examples of higher rank 
were studied; see e.g.~\cite{HK} and references there for disk 
polynomials  as well as  \cite{KS}, \cite{RR},
 \cite{R1}, \cite{R2}, \cite{RV}, 
\cite{V5}, \cite{V6} in the higher rank case,  and \cite{HS}
 for hypergeometric
functions associated with root systems.
Moreover,  series of discrete examples were studied 
 in the setting of trees, graphs,
buildings, association schemes, Hecke pairs, and  other discrete structures;
see e.g.~\cite{B}, \cite{BI}, \cite{CST}, 
\cite{CM}, \cite{DR}, \cite{Kr1}, \cite{L}, \cite{V4} 
where sometimes the
connection to hypergroups is supressed.
This discrete setting  forms the main topic of this paper.
\raggedbottom

Before going into details, 
we return to the general setting. 
Besides  positive product formulas for  $\phi_n(\lambda,.)$  on $D$ 
originating from (\ref{prod-formula-allg}), there exist dual product formulas
for the functions  $\phi_n(.,x)$ ($x\in D$) 
 on  suitable subsets of $\chi(D)$ for the group cases. 
To explain this, 
consider the closed set $P_n(D)\subset\chi(D)$ of
spectral parameters $\lambda$ for which $\phi_n(\lambda,.)\in C(D)$ corresponds
to a positive definite
function on $G_n$. 
For $\lambda_1,\lambda_2\in P_n(D) $, then
$\phi_n(\lambda_1,.)\cdot\phi_n(\lambda_2,.)$  is also positive
definite on $G_n$, which implies that
 $\phi_n(\lambda_1,.)\cdot\phi_n(\lambda_2,.)$ is positive definite on 
the double coset hypergroup $D$.
 Therefore, by Bochner's theorem for commutative hypergroups
 (see \cite{J}), there exists a unique probability measure 
$\mu_{n,\lambda_1,\lambda_2}$ on $P_n(D)$ with the dual product formula
\begin{equation}\label{allg-dual-prod-formula}
\phi_n(\lambda_1,x)\cdot\phi_n(\lambda_2,x)=
\int_{P(D)}\phi_n(\lambda,x)\> d\mu_{n,\lambda_1,\lambda_2}(\lambda)
\quad\quad\text{for all}\quad x\in D.
\end{equation}
For  related results of harmonic analysis for Gelfand pairs and 
commutative  hypergroups we refer to \cite{BH},
\cite{Du}, \cite{F}, \cite{J}.
The dual product formula (\ref{allg-dual-prod-formula}) has an interpretation
in terms of group representations of $G_n$ of class 1, and for several classes
of examples there exist explicite formulas for 
$\mu_{n,\lambda_1,\lambda_2}$ which again can be
extended to a continuous  parameter range $n$ where usually
 the positivity remains available. This is for
instance trivial  for 
 $G_n=SO(n+1)\ltimes\mathbb R^{n+1}, \> H_n=SO(n+1)$, 
with  $P_n(D)=[0,\infty[$,  where, due to the symmetry of 
 $\Lambda_\alpha$ in $x,y$ in (\ref{classical-Bessel-funct}),
 the dual product formula agrees with the
 original one for  $\alpha\in[-1/2,\infty[$. Moreover,
 for
 $G_n=SO(n+2), \> H_n=SO(n+1)$, it is well known that
 $P_n(D)=\chi(D)=\mathbb N_0$,
 and that the dual formula (\ref{allg-dual-prod-formula})
corresponds to the well-known positive product linearization 
for the ultraspherical polynomials  for $\alpha\in]-1/2,\infty[$ which is part
  of a famous explicit positive product linearization formula
 for Jacobi polynomials 
\cite{G1}, \cite{A}.
On the other hand, for hyperbolic spaces, the positive
 dual product formula is  more involved. Here $P_n(D)$  depends on $n$
 (see \cite{FK1}, \cite{FK2}) with $[0,\infty[\subset P_n(D)\subset [0,\infty[
\cup i\cdot  [0,\infty[\subset \mathbb C$.
Moreover, the known explicit formulas for the Lebesgue densities of the
measures
$\mu_{n,\lambda_1,\lambda_2}$    via triple integrals 
for $\lambda_1,\lambda_2\in [0,\infty[$ 
e.g. in  \cite{Ko} are quite involved. 

This picture is  typical for many 
Gelfand pairs. We only mention 
the Gelfand pairs and orthogonal polynomials associated with homogeneous trees
and infinite distance transitive graphs with $D=\mathbb N_0$ and
$\chi(D)=\mathbb C$
in \cite{L}, \cite{V4}, where 
 dual product formulas are computed
 by brute force. For examples of
higher rank, the existence of positive dual product formulas seems to be
 open except for  group cases and   simple self dual examples
like the   Bessel examples  above.
Such self dual examples appear e.g.~as orbit spaces 
when  compact subgroups of $U(n)$
act on $\mathbb R^n$ which leads e.g.~to
examples associated
with matrix Bessel functions  in  \cite{R1}.

In summary, there exist many  continuous families
of commutative hypergroup structures with explicit convolutions, 
for which the multiplicative functions are known  special functions, and 
where the existence of dual product formulas is unknown except for the
group parameters.
The intention of this paper is to start some systematic research  beyond the
group cases  by using more general algebraic structures 
behind commutative hypergroup structures. 

Commutative association schemes  as in \cite{BI}, \cite{B} might form 
such algebraic objects where these
 schemes are defined as finite sets $X\ne\emptyset$
with some partition
 $D$ of $X\times X$ with certain intersection properties;
see Section 3 for details. The most prominent examples appear as homogeneous
spaces $X=G/H$ for subgroups $H$  of  finite groups $G$. It is  known that
all finite association schemes lead to  hypergroup
structures on $D$ which are commutative if and only if so are the
 association scheme. In the group case $X=G/H$, this hypergroup 
is just the double coset hypergroup $G//K$.
There exist
commutative association schemes which do not appear as homogeneous
spaces $X=G/H$ (see \cite{BI}) 
such that the class of commutative hypergroups associated with
 association schemes extends the class of commutative double coset hypergroups.
Moreover, by Section 2.10 of \cite{BI},
these commutative hypergroups  admit  positive dual product formulas.
Therefore, finite commutative association schemes form a tool to
establish dual positive product formulas for some
 finite commutative hypergroups beyond  double coset hypergroups.

The aim of this paper is to show that certain
 generalizations  of commutative association schemes also lead to
commutative hypergroups with dual positive product formulas.
We  first study 
possibly infinite, commutative association schemes where the  theory
of \cite{BI}  can be extended canonically.
This obvious extension 
 is very rigid  as the use of partitions leads to a
very few examples only.
For a further extension, we observe that all association schemes
admit  $\{0,1\}$-valued 
adjacency matrices labeled by $D$ which are 
stochastic   after some renormalization. 
 We  translate the axioms of  association schemes into a
 system of axioms for these matrices. As now the integrality conditions vanish,
we obtain more examples. We show that also in this
generalized case associated commutative hypergroups exist and that, under some
restrictions, dual positive product formulas exist; see Section 5.

We expect that our  approach may be  extended to
non-discrete spaces $X$ where families of Markov kernels 
labeled by some locally compact 
space $D$ instead of stochastic  matrices  are used.
We shall study this non-discrete generalization  in a forthcoming paper.
It covers the group cases with $X=G/H$, $D=G//H$ for Gelfand pairs $(G,H)$
as well two known continuous series of examples 
with $X=\mathbb R^2$, $D=[0,\infty[$ and $X=S^2$ the 2-sphere in $\mathbb R^3$,
 $D=[-1,1]$ due to Kingman \cite{Ki} and Bingham \cite{Bin}. The associated 
commutative hypergroups on $D$ with dual positive product formulas
will be just the hypergroups associated with Bessel functions and
ultraspherical polynomials mentioned above.
Even if for these examples the dual positive product formulas are well known,
 these examples might be a hint that generalized continuous commutative 
association schemes might form a powerful tool 
to derive the existence of  dual positive product formulas.
We hope that this approach can be applied to
certain  Heckman-Opdam 
Jacobi polynomials of type BC and hypergeometric functions of type BC 
(see \cite{HS}), 
which generalize the spherical functions of compact and noncompact
Grassmannianss and for which continuous families of commutative 
hypergroup structures exist by \cite{RR} and \cite{R2}.

This paper is organized as follows. 
In Section 2 we recapitulate some facts about commutative hypergroups.
Section  3 is then devoted to possibly infinite association schemes and 
 associated hypergroups where 
the definition  remains very close to 
the classical one in \cite{BI}.
We there in particular study the relations to the double coset hypergroups $G//H$ for
compact open subgroups $H$ of $G$ and for Hecke pairs $(G,H)$.
 In Section 4 we prove that for all commutative hypergroups 
associated with such  association schemes
there exist positive dual convolutions and dual product 
formulas at least on the support of the Plancherel measure. 
In Section 5 we  propose a discrete generalization of association 
schemes without integrality conditions.
We show that under some conditions, 
many results of Sections 3 and 4 remain valid  including the
existence of  dual product formulas.

\section{Hypergroups}

Hypergroups  form an extension  of locally compact   groups. For
 this,  remember that the group multiplication
 on a locally compact   group $G$ leads to the convolution
$\delta_x*\delta_y=\delta_{xy}$ ($x,y\in G$) of point measures. 
Bilinear, weakly continuous extension  of this convolution 
then leads to a  
 Banach-$*$-algebra structure on the Banach space $M_b(G)$ of all
signed bounded regular Borel measures with the total variation norm
$\|.\|_{TV}$.
In the case of hypergroups we only require 
 a convolution  $*$ for measures which admits 
most properties of a group
 convolution. We here recapitulate some well-known facts; for 
 details see \cite{Du}, \cite{J}, 
 \cite{BH}.

\begin{definition}
A hypergroup $(D,*)$ is a locally compact Hausdorff space $D$ with a
  weakly continuous, associative, bilinear convolution $*$ on the Banach space $M_b(D)$
  of all bounded regular Borel measures with the following properties:
\begin{enumerate}
\item[\rm{(1)}] For all $x,y\in D$,  $\delta_x*\delta_y$  is a compactly
  supported probability measure on $D$ such that the support 
  $\text{supp}\>(\delta_x*\delta_y)$ depends continuously on $x,y$ w.r.t.~the
  so-called Michael topology on the space of all compacta in $X$
 (see \cite{J} for details).  
\item[\rm{(2)}] There exists a neutral element $e\in D$ with $\delta_x*\delta_e=
\delta_e*\delta_x=\delta_x$ for  $x\in D$. 
\item[\rm{(3)}] There exists a continuous involution $x\mapsto\bar x$ on $X$
  such that for all $x,y\in D$, $e\in \text{supp}\> (\delta_x*\delta_y)$ holds
  if and only if $y=\bar x$. 
\item[\rm{(4)}] If for $\mu\in M_b(D)$, $\mu^-$ denotes the
  image of $\mu$ under the involution, then
  $(\delta_x*\delta_y)^-= \delta_{\bar y}*\delta_{\bar x}$ for all $x,y\in D$.
\end{enumerate}
A hypergroup is called commutative if the convolution $*$ is commutative. 
It is called symmetric if the involution is the identity. 
\end{definition}

\begin{remarks}
\begin{enumerate}
\item[\rm{(1)}]  The identity $e$ and the involution $.^-$ above
are  unique.
\item[\rm{(2)}] Each symmetric hypergroup is  commutative.
\item[\rm{(3)}] For each hypergroup $(D,*)$, $(M_b(D),*)$ is a
  Banach-$*$-algebra with the involution $\mu\mapsto\mu^*$ with
$\mu^*(A):=\overline{\mu(A^-)}$   for Borel sets $A\subset D$.
\item[\rm{(4)}] For a second countable locally compact space $D$, the Michael
  topology agrees with the well-known Hausdorff topology; see \cite{KS}.
\end{enumerate}
\end{remarks}

The most prominent examples are  
double coset hypergroups $G//H:=\{HgH:\> g\in G\}$
for compact subgroups $H$  of 
locally compact groups $G$:

\begin{example}\label{Doublecoset} Let $H$ be  a compact
 subgroup of a locally compact group
  $G$ with identity $e$ and with
the unique normalized Haar measure $\omega_H\in M^1(H)\subset M^1(G)$, i.e.
$\omega_H$ is a probability measure.
Then the space
$$M_b(G||H):=\{\mu\in M_b(G):\> \mu=\omega_H*\mu*\omega_H\}$$
of all $H$-biinvariant measures in $M_b(G)$ is a Banach-$*$-subalgebra of
$M_b(G)$. With the quotient topology, 
 $G//H$  is a locally compact space, and the
canonical
projection $p_{G//H}:G\to G//H$ is continuous, proper and open. 
Now consider the push forward (or image-measure mapping)
  $\tilde p_{G//H}:M_b(G)\to M_b(G//H)$  with 
$\tilde p_{G//H}(\mu)(A)=\mu(p_{G//H}^{-1}(A))$ for $\mu\in M_b(G)$
and Borel sets $A\subset G//H$. It is easy to see that 
 $\tilde p_{G//H}$ is  an isometric isomorphism between the Banach
 spaces  $M_b(G||H)$ and $M_b(G//H)$ w.r.t.~the total variation
 norms, and that the  transfer of the  convolution on
 $M_b(G||H)$  to   $M_b(G//H)$
 leads to a hypergroup $(G//H, *)$ with  identity $HeH$ 
 and  involution $HgH\mapsto Hg^{-1}H$. For details see  \cite{J}.
\end{example}

Let us consider some typical discrete double coset hypergroups $G//H$:

\begin{example}\label{graphauto}
Let $\Gamma$ be the vertex set  of a  locally finite,
connected  undirected graph with the graph metric $d:\Gamma\times\Gamma\to
\mathbb N_0:=\{0,1,\ldots\}$. A bijective mapping $g:\Gamma\to\Gamma$ is
called an automorphism of $\Gamma$ if 
$d(g(a),g(b))=d(a,b)$ for all $a,b\in\Gamma$.
Clearly, the set  $Aut(\Gamma)$ of all automorphisms is a topological group w.r.t.~the topology
of pointwise convergence, i.e., we regard $Aut(\Gamma)$  as  subspace of
$\Gamma^\Gamma$  equipped with the product topology.
Assume now that  $Aut(\Gamma)$ 
acts transitively on  $\Gamma$. 
It is well-known and easy to see that
  $Aut(\Gamma)$  is a totally disconnected locally compact group
which contains the stabilizer subgroup $H_x\subset G$ of any $x\in \Gamma$  as
a compact open  subgroup.
  $\Gamma$ can be identified with
 $Aut(\Gamma)/H_x$, and
 the discrete orbit  space $\Gamma^{H_x}:=\{H_x(y):\> y\in
\Gamma\}$   with the
discrete double coset space $Aut(\Gamma)//H_x$.
\end{example}

We  also consider another kind of double coset hypergroups:

\begin{example}\label{hecke1}
Let $G$ be a discrete group with some subgroup $H$.
$(G,H)$ is called a Hecke pair if the so-called Hecke condition holds, i.e., if
each double coset $HgH$ ($g\in G$) decomposes into an at most finite number of 
right cosets $g_1H,\ldots, g_{ind(HgH)}H$ where  $ind(HgH)\in\mathbb
N$ is called the (right-)index of $HgH$. 
Hecke pairs are studied e.g.~in \cite{Kr1}, \cite{Kr2} where  left coset are taken.

For  Hecke pairs,  $G//H$ carries a discrete hypergroup
structure due to \cite{Kr2}. To describe the  associated convolution,
take $a,b\in G$ and consider the disjoint decompositions
$HaH=\cup_{i=1}^n a_iH$, $HbH=\cup_{j=1}^m b_jH$. If we put 
$$\mu(HcH):=|\{(i,j):\> a_ib_jH= cH\}|\in\mathbb N_0,$$
then $\mu(HcH)$ is independent of the representative $c$ of $HcH$, and
$$\delta_{HaH}*\delta_{HbH}:= \sum_{HcH\in G//H}
\frac{\mu(HcH)\cdot ind(HcH)}{ ind(HaH) \cdot ind(HbH)}\delta_{HcH}
\quad\quad(a,b\in G)$$
generates a hypergroup structure on $G//H$.
\end{example}

The notion of  Haar
measures on hypergroups is  similar to  groups:

\begin{definition} Let $(D,*)$ be a hypergroup,
 $x,y\in D$, and $f\in C_c(D)$ a continuous function with compact support.
We write $_xf(y):=f(x*y):=\int_K f\>
d(\delta_x*\delta_y)$ and $f_x(y):=f(y*x)$  where, 
by the hypergroup axioms, $f_x,\>_xf  \in C_c(D)$ holds.

A non-trivial positive  Radon measure $\omega\in M^+(D)$ 
is called a left or right Haar measure if 
$${\int_D} {_xf}\> d\omega=
\int_D f\> d\omega     \quad\text{or}\quad
{\int_D} {f_x}\> d\omega=
\int_D f\> d\omega \quad\quad (f\in C_c(D), \> x\in D)$$ 
  respectively. $\omega$ is 
 called a Haar measure if
it is a left and  right Haar measure.
If $(D,*)$  admits a  Haar measure, then it is called unimodular. 
\end{definition} 

The uniqueness of left and right  Haar measures
 and their existence for
particular classes  are known for a long time by 
Dunkl, Jewett, and Spector;
see \cite{BH} for details. The general existence was settled only
 recently by Chapovsky \cite{Ch}:

\begin{theorem} Each hypergroup admits a left and a right Haar measure.
Both  are unique up to normalization.
\end{theorem}

\begin{examples}\label{example-haar}
\begin{enumerate} 
\item[\rm{(1)}] Let  $(D,*)$ be a discrete hypergroup. Then, by \cite{J},
 left and right Haar measures
 are given by 
$$\omega_l(\{x\})= \frac{1}{(\delta_{\bar x}*\delta_x)(\{e\})}, \quad
\omega_r(\{x\})= \frac{1}{(\delta_{ x}*\delta_{\bar x})(\{e\})}, \quad
\quad(x\in D).$$
\item[\rm{(2)}] If $(G//H,*)$ is a double coset hypergroup and  $\omega_G$
 a left Haar measure of $G$, then its projection to
 $G//H$ is a left Haar measure of $(G//H,*)$. 
 \end{enumerate}
\end{examples}

We next recapitulate some facts on Fourier analysis on commutative hypergroups
from \cite{BH}, \cite{J}.
For the rest of Section 2 let $(D,*)$ be a commutative hypergroup with Haar
measure $\omega$. For $p\ge1$ consider the $L^p$-spaces
$L^p(D):=L^p(D,\omega)$. Moreover $C_b(D)$ and $C_o(D)$ are the Banach spaces
of all bounded continuous functions on $D$ and 
those which vanish at infinity respectively.

\begin{definition}
\begin{enumerate} 
\item[\rm{(1)}] The dual space of $(D,*)$ is defined as
$$\hat D:=\{\alpha\in C_b(D):\>\> \alpha\not\equiv 0,\>\>
\alpha(x*\bar y)=\alpha(x)\cdot \overline{\alpha(y)} \>\>\text{for all}\>\>
x,y\in D\}.$$
$\hat D$ is a locally compact space w.r.t.~the topology of compact-uniform
convergence. Moreover, if $D$ is discrete, then $\hat D$ is compact. All
characters $\alpha\in\hat D$ satisfy $\|\alpha\|_\infty=1$ with $\alpha(e)=1$.
\item[\rm{(2)}] For $f\in L^1(D)$ and $\mu\in M_b(D)$, the Fourier transforms
  are defined by
$$\hat f(\alpha):=\int_D f(x)\overline{\alpha(x)}\> d\omega(x),\quad
\hat \mu(\alpha):=\int_D \overline{\alpha(x)}\> d\mu(x) 
\quad (\alpha\in\hat D)$$
with $\hat f\in C_o(\hat D)$, $\hat\mu\in C_b(\hat D)$ and 
$\|\hat f\|_\infty\le\|f\|_1$, $\|\hat \mu\|_\infty\le\|\mu\|_{TV}$.
\item[\rm{(3)}] There exists a unique positive measure $\pi\in M^+(\hat D)$
  such that the  Fourier transform
$.^\wedge: L^1(D)\cap L^2(D) \to C_0(\hat D)\cap L^2(\hat D,\pi)$
is an isometry. $\pi$ is called the Plancherel measure on $\hat D$.

Notice that, different from l.c.a.~groups, the support
$S:=supp\> \pi$ may be a proper closed subset of $\hat D$.
Quite often, we even have ${\bf 1}\not\in S$.
\item[\rm{(4)}] For $f\in L^1(\hat D,\pi)$, $\mu\in M_b(\hat D)$, the
  inverse
Fourier transforms are given by
$$\check f(x):=\int_S f(\alpha) {\alpha(x)}\> d\pi(\alpha),\quad
\check \mu(x):=\int_{\hat D} {\alpha(x)}\> d\mu(\alpha) 
\quad (x\in D)$$
with $\check f\in C_0(D)$, $\check\mu\in C_b(D)$
and 
$\|\check f\|_\infty\le\|f\|_1$, $\|\check \mu\|_\infty\le\|\mu\|_{TV}$.
\item[\rm{(5)}] $f\in C_b(D)$ is called positive definite 
on the hypergroup $D$ if for all $n\in\mathbb N$, $x_1,\ldots, x_n\in D$ and
  $c_1,\ldots,c_n\in\mathbb C$,
$\sum_{k,l=1}^n c_k \bar c_l \cdot f(x_k*\bar x_l)\ge0.$
Obviously, all characters $\alpha\in\hat D$ are positive definite.
 \end{enumerate}\end{definition}

We collect some essential well-known results:

\begin{facts}\label{facts-hy}
\begin{enumerate} 
\item[\rm{(1)}] (Theorem of Bochner, \cite{J}) A function $f\in C_b(D)$ is positive
  definite if and only of $f=\check \mu$ for some $\mu\in M_b^+(\hat D)$.
In this case, $\mu$ is a probability measure if and only if
$\check\mu(e)=1$.
\item[\rm{(2)}] For $f,g\in L^2(D)$, the convolution product 
$f*g(x):=\int f(x*\bar  y)g(y)\> d\omega(y)$ ($x\in D$)
 satisfies $f*g\in C_0(D)$. Moreover,
for $f\in L^2(D)$,  $f^*(x)=\overline{f(\bar x)}$ satisfies
  $f^*\in L^2(D)$, and 
 $f*f^*\in C_0(D)$ is positive definite; see \cite{J}, \cite{BH}. 
\item[\rm{(3)}] A function $f\in C_b(D)$ is the inverse Fourier transform
  $\check\mu$ for some $\mu\in M_b^+(\hat D)$ with $supp\>\mu\subset S$ if and
  only if $f$ is the compact-uniform limit of positive definite functions of
  the form $h*h^*$, $h\in C_c(D)$; see \cite{V2}.
\item[\rm{(4)}] Let $\alpha\in\hat D$. Then $\alpha\in S$ if and only if 
 $\alpha$ is the compact-uniform limit of positive definite functions of
  the form $h*h^*$, $h\in C_c(D)$; see \cite{V2}.
\item[\rm{(5)}] There exists precisely one positive character $\alpha_0\in S$
by \cite{V1}, \cite{BH}.
\item[\rm{(6)}] If $\mu\in M^1(\hat D)$ satisfies $\check\mu\ge0$ on $D$, then 
its support $supp\>\mu$ contains at least one positive character; see \cite{V3}. 
\end{enumerate}
\end{facts}

In contrast to l.c.a.~groups, products of positive definite
functions on $D$ are not necessarily positive definite; see e.g.~Section
9.1C of \cite{J} for an example with $|D|=3$. However,sometimes
 this positive definiteness of products is
available. 

If for all $\alpha,\beta\in \hat D$ (or a subset of  $\hat D$ like
$S$) the products $\alpha\beta$ are positive definite, 
then by Bochner's theorem \ref{facts-hy}(1), there are probability measures 
$\delta_\alpha\hat*\delta_\beta\in M^1(\hat D)$ with 
$(\delta_\alpha\hat*\delta_\beta)^\vee=\alpha\beta$, i.e., we obtain dual 
positive product formulas as claimed in Section 1. Under additional
conditions, $(\hat D,\hat*)$ then  carries a dual hypergroup structure
with ${\bf 1}$ as identity and complex conjugation as involution. This  for
instance holds for all compact commutative double coset hypergroups $G//H$
by  \cite{Du}. For non-compact Gelfand pairs $(G,H)$
 we have  dual positive  
convolutions on $S$; see \cite{J}, \cite{V3}. These convolutions usually do not
generate a dual hypergroup structure, and sometimes $\alpha\beta$
is not positive definite on $D$ for some  $\alpha,\beta\in\hat D$; see Theorem 
\ref{podi} for an example.

\section{Discrete association schemes}

In this section we extend the classical notion of finite association schemes 
 in a
natural way. As 
references for finite association schemes 
 we recommend \cite{BI}, \cite{B}.

\begin{definition}\label{def-discrete-asso}
 Let $X, D$ be nonempty, at most countable 
sets and $(R_i)_{i\in D}$ a disjoint
  partition of $X\times X$ with $R_i\ne\emptyset$ for $i\in D$ and the
  following properties:
\begin{enumerate}
\item[\rm{(1)}] There exists $e\in D$ with $R_e=\{(x,x): x\in X\}$.
\item[\rm{(2)}] There exists an involution $i\mapsto \bar i$ on $D$ such that
  for $i\in D$,
$R_{\bar i}=\{(y,x):\> (x,y)\in R_i\}$.
\item[\rm{(3)}] For all $i,j,k\in D$ and $(x,y)\in R_k$, the number
$$p_{i,j}^k:=|\{z\in X:\> (x,z)\in R_i\>\> {\rm and }\>\> (z,y)\in R_j\}|$$
is finite and independent of $(x,y)\in R_k$.
\end{enumerate}
Then $\Lambda:=(X,D, (R_i)_{i\in D})$ is called an association scheme with
intersection numbers $(p_{i,j}^k)_{i,j,k\in D}$ and identity $e$. 

An association scheme is called commutative if $p_{i,j}^k=p_{j,i}^k$
 for all $i,j,k\in D$. It is called symmetric (or hermitian) if
the involution on $D$ is the identity.
Moreover, it is called finite, if so are $X$ and $D$.
\end{definition}

Finite association schemes above are obviously precisely 
association schemes in the sense of the monographs \cite{BI}.

Association schemes have the following interpretation: We
regard $X\times X$ as set of all directed paths from points in $X$
 to points in $X$. The set of paths is labeled by $D$ which might be colors,
 lengths, or difficulties of  paths. 
Then $R_e$ is the set of trivial paths, and
 $R_{\bar i}$ is the set of all reversed paths in $R_i$. Axiom (3) 
 is  some kind of a symmetry condition and  the central part of the
 definition.

Association schemes may be described with the aid of adjacency matrices:

\begin{definition}\label{def-adjacency-matrix}
The  adjacency 
matrices $A_i\in \mathbb R^{X\times X}$ ($i\in D$) of an association scheme 
 $(X,D, (R_i)_{i\in D})$ are given by
$$(A_i)_{x,y}:=
\left\{ \begin{array}{cc}
\displaystyle 1 &
 {\rm if}\>\> (x,y)\in R_i
\\     0&  {\rm otherwise}\\     \end{array} \right. 
 \quad\quad\quad \quad(i\in D,\>
x,y\in X).$$
The adjacency matrices  have the following obvious properties: 
\begin{enumerate}
\item[\rm{(1)}] $A_e$ is the identity matrix $I_X$.
\item[\rm{(2)}] $\sum_{i\in D} A_i$ is the matrix $J_X$ whose entries are all
  equal to 1.
\item[\rm{(3)}] $A_i^T = A_{\bar i}$ for $i\in D$.
\item[\rm{(4)}] For all $i\in D$ and all rows and columns of $A_i$, all
  entries are equal to
  zero except for finitely many cases (take $k=e$, $j=\bar i$ in \ref{def-discrete-asso}(3)!).
\item[\rm{(5)}] For $i,j\in D$, 
 $A_iA_j= \sum_{k\in D} p_{i,j}^k A_k$.
\item[\rm{(6)}] An association scheme  $\Lambda$ is commutative if and only
  if $A_iA_j=A_jA_i$ for all  $i,j\in D$.
\item[\rm{(7)}] An association scheme  $\Lambda$ is symmetric if and
  only if  all $A_i$ are symmetric. In particular, each  symmetric
 association scheme is commutative.
\end{enumerate}
\end{definition}

\begin{definition}\label{def-valency}
Let $(X,D, (R_i)_{i\in D})$  be an  association scheme.
 The valency of $R_i$ or $i\in D$ is
defined as  $\omega_i:=p_{i,\bar i}^e$. Obviously,
 the $\omega_i$ satisfy
\begin{equation}\label{deutung-omegai}
\omega_i=|\{z\in X:\> (x,z)\in R_i\}|\in\mathbb N\end
{equation}
for all $x\in X$.
In particular,  $\omega_e=1$,
 and
\begin{equation}\label{valency-summation-all}
\textstyle |X|=\sum_{i\in K}\omega_i\in\mathbb N\cup\{\infty\}.
\end{equation}
\end{definition}

\begin{remark}\label{remark-asso-stoch}
For   $i\in D$, the  renormalized matrices
  $S_i:= \frac{1}{\omega_i}A_i\in \mathbb R^{X\times X}$ are stochastic, i.e.,
  all  rows sum are equal to 1 by (\ref{deutung-omegai}).
Notice that also all column  sums are finite by
\ref{def-adjacency-matrix}(4). Moreover, by
\ref{def-adjacency-matrix}(5), the stochastic matrices $S_i$ satisfy
\begin{equation}\label{rel-stoch-matrices}
\textstyle S_iS_j= \sum_{k\in D} \frac{ \omega_k}{\omega_i\omega_j}
p_{i,j}^k S_k  \quad\quad\text{ for}  \quad\quad i,j\in D.
\end{equation}
Since  both sides of (\ref{rel-stoch-matrices}) are stochastic,
 absolute convergence leads to
\begin{equation}\label{sum-one}
\textstyle
\sum_{k\in D}\frac{ \omega_k}{\omega_i\omega_j} p_{i,j}^k=1 \quad\text{ for}  \quad i,j\in D.
\end{equation}

The formulas (\ref{rel-stoch-matrices}) and (\ref{sum-one}) are the
starting point in Section 5 for the extension of association schemes to
a continuous setting.
\end{remark}

We now collect some relations about the intersection numbers and
valencies:

\begin{lemma}\label{multass} For all $i,j,k, l, m\in D$:
\begin{enumerate}
\item[\rm{(1)}] $p_{e,i}^j=p_{i,e}^j=\delta_{i,j}$ and 
$p_{i,j}^e= \omega_i\delta_{i,\bar j}$.
\item[\rm{(2)}] $p_{i,j}^l =p_{\bar j,\bar i}^{\bar l}$.
\item[\rm{(3)}] $\sum_{j\in K} p_{i,j}^l =\omega_i$ for all $l\in D$, and, in
  particular, for all $i,l\in D$, $p_{i,j}^l\ne 0$ holds for finitely many
  $j\in K$ only.
\item[\rm{(4)}] $\omega_l\cdot  p_{i,j}^l=\omega_i\cdot  p_{l,\bar j}^i$ and 
$\omega_{\bar j}\cdot p_{\bar i, l}^j = \omega_{\bar l}\cdot  p_{i,j}^l$.
\item[\rm{(5)}] $\sum_{l\in K} \omega_l\cdot  p_{i,j}^l= 
\omega_i\cdot\omega_j\>\>$ and
  $\>\>\sum_{l\in K} \omega_{\bar l}\cdot  p_{i,j}^l=
 \omega_{\bar i}\cdot\omega_{\bar j}$
\item[\rm{(6)}] $\sum_{l\in K}  p_{i,j}^l\cdot  p_{l,k}^m = 
\sum_{l\in K}  p_{j,k}^l p_{l,k}^m$.
\item[\rm{(7)}] If $p_{i,j}^k>0$, then 
$\displaystyle \frac{\omega_k}{\omega_{\bar k}} = 
 \frac{\omega_i}{\omega_{\bar i}}\cdot  \frac{\omega_j}{\omega_{\bar j}}$.
\end{enumerate}
\end{lemma}

\begin{proof} Part (1) is obvious, and  part (2) follows from 
$$\textstyle\sum_l  p_{\bar j,\bar i}^{\bar l} A_{\bar l} = A_j^TA_i^T=(A_iA_j)^T=
\sum_l p_{i,j}^l A_{l}^T= \sum_l p_{i,j}^l A_{\bar l }. $$
and a comparison of coefficients. Part (3) is a consequence of
$$\textstyle
\omega_i J_X = A_i J_X = \sum_j A_i A_j =  \sum_j \sum_l  p_{i,j}^l A_l = 
\sum_l \Bigl(\sum_{j}   p_{i,j}^l\Bigr) A_l.$$
For the proof of the first  statement of  (4), notice that by
part (1), for  $x\in X$
$$\omega_l \cdot  p_{i,j}^l= \sum_k   p_{i,j}^k  (A_k A_{\bar l})_{x,x}=
(A_iA_jA_{\bar l})_{x,x} = \sum_k p_{j,\bar l}^k (A_iA_k)_{x,x}=
 \omega_i\cdot p_{j,\bar l}^{\bar i}= \omega_i\cdot p_{l,\bar j}^i;$$
the second  statement of part (4) follows in a similar way. 
Moreover, the statements in (5) are consequences of
$$\textstyle \omega_i\cdot\omega_j J_X = A_iA_j J_X = \sum_l  p_{i,j}^l A_l J_X=
\sum_l  p_{i,j}^l\cdot\omega_l J_X$$
and
$$\textstyle \omega_{\bar i}\cdot\omega_{\bar j} J_X = J_XA_iA_j  =
 \sum_l  p_{i,j}^lJ_X A_l=
\sum_l  p_{i,j}^l\cdot\omega_{\bar l} J_X.$$
Part (6) follows from $(A_iA_j)A_k=A_i(A_jA_k)$ and comparison of 
coefficients in the expansions. Finally, parts (4) and (2) imply
$$ \frac{\omega_k}{\omega_i\cdot \omega_j}  p_{i,j}^k  = 
 \frac{ p_{k, \bar j}^i}{\omega_j}  =
  \frac{ p_{\bar k, i}^{\bar j}}{\omega_{\bar i}}=
 \frac{ p_{\bar i, k}^{ j}}{\omega_{\bar i}}=
 \frac{\omega_{\bar k}}{\omega_{\bar i}\omega_{\bar j}}   p_{i,j}^k, $$
which proves (7).
\end{proof}

Typical examples of  finite or infinite association schemes are connected
with homogeneous spaces $G/H$ in one of the following two ways: 

\begin{example}\label{example-totally-disconnected}
Let $G$ be a second countable, locally compact
 group with an compact open subgroup $H$ and 
 with neutral element $e$.
 Then the quotient $X:=G/H$ as
well as the double coset space $D:=G//H$ are at most countable, discrete spaces
w.r.t.~the quotient topology. Consider the partition $(R_i)_{i\in D}$ of
$X\times X$ with
$$R_{HgH}:=\{(xH,yH)\in X\times X:\> HgH=Hx^{-1}yH\}.$$
Then $(X,D,(R_{j})_{j\in D})$ forms an association scheme with neutral element
 $HeH=H$ and involution $\overline{HgH}:=Hg^{-1}H$. In fact, the axioms (1),
  (2) in \ref{def-discrete-asso} are obvious. For axiom (3) consider
  $g,h,x,y,\tilde x, \tilde y\in G$ with $(xH,yH)$ and
 $(\tilde xH, \tilde yH)$ in the same partition $R_{Hx^{-1}yH}$, i.e., with
$H\tilde x^{-1}\tilde yH=Hx^{-1}yH$. Thus, there exist $h_1,h_2\in H$ and
  $w\in G$ with $\tilde x^{-1}=h_1 x^{-1}w^{-1}$ and $\tilde  y=wyh_2$.
 Therefore, for any $zH\in X$, $H\tilde x^{-1}zH=Hx^{-1}w^{-1}zH$
  and $Hz^{-1}wyH=H(w^{-1}z)^{-1}H$, which means that $zH\mapsto w^{-1}zH$
  establishes a bijective mapping between 
$\{zH:\> Hx^{-1}zH=HgH, \> Hz^{-1}yH=HhH\}$ and
$\{zH:\> H\tilde x^{-1}zH=HgH, \> Hz^{-1}\tilde yH=HhH\}$. This shows that
  the intersection number $p_{HgH, HhH}^{H x^{-1}yH}$ is independent of
  the choice of $x,y$. Moreover, due to compactness, each double coset $HgH$
  decomposes into finitely many cosets $xH$ which shows that 
$\{zH:\> H x^{-1}zH=HgH\}$ and thus the intersection number $p_{HgH,
    HhH}^{H x^{-1}yH}$  is finite as claimed.

It follows from the definition of $R_{HgH}$ and 
Eq.~(\ref{deutung-omegai}) that for $g\in G$
 the valency $\omega_{HgH}$ of the  double coset $HgH\in G//H$ is given
by the finite number of different cosets $xH\in G/H$
 contained in $HgH$.
\end{example}

The preceding example also works  for  Hecke pairs:

\begin{example}\label{example-Hecke-condition}
Let $(G,H)$ be a  Hecke pair, i.e., 
each double coset $HgH$  decomposes into finitely
many cosets $xH$. Then
the same partition as in Example 
\ref{example-totally-disconnected} leads to an association scheme 
 $(G/H,G//H,(R_{j})_{j\in D})$. We skip the  proof.
 We remark that again the valency
 $\omega_{HgH}$ of a double coset $HgH\in G//H$ is $ind(HgH)$, i.e.,
 the  number of  cosets $xH\in G/H$ contained in $HgH$.
\end{example}

Association schemes lead to discrete
hypergroups. The associated convolution algebras are just the Bose-Mesner
algebras for finite association schemes in \cite{BI}.

\begin{proposition}\label{prp-asso-to-hypergroup}
 Let $\Lambda:=(X, D, (R_i)_{i\in D})$ be an association
  scheme with intersection numbers $p_{i,j}^k$ and valencies $\omega_i$.
Then the product $*$ with
$$\delta_i*\delta_j:=\sum_{k\in D}\frac{\omega_{ k}}{\omega_{ i}\omega_{ j}} 
\cdot p_{i,j}^k \delta_k$$
can be extended uniquely to an associative, bilinear, 
$\|\>.\>\|_{TV}$-continuous mapping on $M_b(D)$.
 $(D,*)$  is a discrete hypergroup with the left and right Haar measure
\begin{equation}\label{const-haar}
\Omega_l:=\sum_{i\in D} \omega_i \delta_i
\quad\quad\text{and}\quad\quad \Omega_r:=\Omega_l^*:=
\sum_{i\in D} \omega_{\bar i} \delta_i   
\end{equation}
 respectively. This  hypergroup is commutative or  
 symmetric  if and only if so is $\Lambda$.
\end{proposition}

\begin{proof} By Lemma \ref{multass}(5), $\delta_i*\delta_j$ is a probability
  measure with finite support for  $i,j\in D$. Thus, $*$ can be
  extended uniquely to a bilinear continuous mapping on $M_b(D)$. 
The associativity of $*$  follows from  Lemma \ref{multass}(6) for
point measures, and, in the general case, by the unique  bilinear continuous
extension. The remaining hypergroup axioms now follow from
  Lemma \ref{multass}. Clearly, $*$ is commutative if and only if so is 
 $\Lambda$. The same holds for symmetry, as the involution on a
  hypergroup and on an association scheme are unique. 

Finally, by  Example
\ref{example-haar}, a left Haar measure is given by
 $$\Omega_l(\{i\}) = \frac{1}{\tilde p_{\tilde i, i}^e}= 
\frac{\omega_{\tilde i}\omega_{ i}}{\omega_e  p_{\tilde i, i}^e}=\omega_i
\quad(i\in D).$$
The same  argument works for  right Haar measures.
\end{proof}

\begin{example}\label{equal-const}
 Let $G$ be a second countable, locally compact group with
  a compact open subgroup $H$. Consider the associated quotient 
association scheme 
 $(X=G/H,D=G//H,(R_{j})_{j\in D})$ as in Example
\ref{example-totally-disconnected}. Then the 
associated hypergroup $(D=G//H,*)$ 
according to Proposition \ref{prp-asso-to-hypergroup} is  the double
coset hypergroup in the sense of Section 2.

For the proof  notice that  both hypergroups  live on the discrete space $D$.
We must show
that for all $x,y,g\in G$ the products $(\delta_{HxH}*\delta_{HyH})(\{HgH\})$
are equal.

Let us compute this first for the double coset hypergroup $G//H$:
By Example \ref{example-totally-disconnected}, $HxH$ decomposes
into $\omega_{HxH}$ many disjoint cosets $x_1H,\ldots, x_{\omega_{HxH}}H$.
 This shows
that $HxH$ also decomposes into  $\omega_{Hx^{-1}H}$  many disjoint cosets
of the form  $H\tilde x_1^{-1},\ldots, H \tilde x_{\omega_{Hx^{-1}H}}^{-1}$. Moreover,
 $HyH$ decomposes
into $\omega_{HyH}$ many disjoint cosets $y_1H,\ldots, y_{\omega_{HyH}}H$.
Using the normalized Haar measure $\omega_H$ of $H$ we obtain
\begin{align}((\omega_H*\delta_x*\omega_H)&*(\omega_H*\delta_y*\omega_H))(HgH)=
\\&=
\frac{1}{\omega_{Hx^{-1}H}\cdot \omega_{HyH}}
\sum_{k=1}^{\omega_{Hx^{-1}H}}\sum_{l=1}^{\omega_{HyH}}
(\omega_H*\delta_{\tilde x_k^{-1}}*\delta_{y_l}*\omega_H)(HgH).\notag
\end{align}
Therefore, by the definition of the double coset convolution in Section 2,
\begin{equation}\label{doppelnebenkl-deutung}
(\delta_{HxH}*\delta_{HyH})(\{HgH\})=
\frac{|\{(k,l): H\tilde x_k^{-1}y_lH=HgH\}|}{\omega_{Hx^{-1}H} \cdot\omega_{HyH}}.
\end{equation}

We next check that the hypergroup associated with the association scheme
in Example \ref{example-totally-disconnected} leads to the same result.
For this we employ  the beginning of the proof of \ref{multass}(4) and observe
that
$$\omega_{HgH}\cdot p_{HxH, HyH}^{HgH}=(A_{HxH}A_{HyH}A_{Hg^{-1}H})_{eH,eH}$$ for the
trivial coset $eH\in X=G/H$. This shows  with the notations above that
$$\omega_{HgH}\cdot p_{HxH, HyH}^{HgH}=|\{(k,l): H\tilde x_k^{-1}y_lH=HgH\}|.$$
The convolution in Proposition 
\ref{prp-asso-to-hypergroup} now again leads to (\ref{doppelnebenkl-deutung}) as claimed.
\end{example}

The same result can be obtained for  Hecke pairs.
We  omit the obvious proof.

\begin{example}
Let $(G,H)$ be a Hecke pair. Consider the associated quotient 
association scheme 
 $(X=G/H,D=G//H,(R_{j})_{j\in D}$ as  in Example
\ref{example-Hecke-condition}. Then the 
associated hypergroup $(D=G//H,*)$ 
according to Proposition \ref{prp-asso-to-hypergroup} is the double
coset hypergroup in the sense of Section 2.
\end{example}

We next discuss a property of association schemes which is valid in the
finite case,  but not necessarily in infinite cases.

\begin{definition} An association scheme $\Lambda$ with
 valencies $\omega_i$ is called unimodular if $\omega_i=\omega_{\bar i}$
 for all $i\in D$. 
\end{definition}

\begin{lemma}\label{unimod}
\begin{enumerate}
\item[\rm{(1)}] If an association scheme is commutative or finite,
 then it is unimodular.
\item[\rm{(2)}] An association scheme $(X,D,(R_i)_{i\in D})$ is unimodular if
  and only if
 the associated discrete hypergroup $(D,*)$ is unimodular.
\item[\rm{(3)}] If  $(X,D,(R_i)_{i\in D})$ is unimodular, then
$S_i^T=S_{\bar i}\quad (i\in D).$
\end{enumerate}
\end{lemma}

\begin{proof}
The commutative case in (1) is trivial. Now let the set $X$ of an association
scheme finite. Then, for $h\in D$, the adjacency matrices $A_h$ and $A_{\bar
  h}$ have $|X|\cdot\omega_h$ and  $|X|\cdot\omega_{\bar h}$ times the entry 1
respectively. $A_{\bar h}=A_h^T$ now implies $\omega_h=\omega_{\bar h}$.

 Part (2) is clear by Eq.~(\ref{const-haar})
 in Proposition \ref{prp-asso-to-hypergroup}; (3) is also clear.
\end{proof}

\begin{remark}
\begin{enumerate}
\item[\rm{(1)}] There exist  non-unimodular association schemes  in
the group context \ref{example-totally-disconnected}. In fact, in
\cite{KW}, examples of totally disconnected groups $G$ with compact open
subgroups $H$ are given where the associated double coset hypergroups
$(G//H,*)$ are not unimodular. Therefore, by Lemma \ref{unimod}(2),
the corresponding association schemes are not unimodular. For examples 
in the context of Hecke pairs see \cite{Kr1}.
\item[\rm{(2)}] Lemma \ref{multass}(7) has the interpretation that the modular
  function $\Delta(i):=\frac{\omega_i}{\omega_{\bar i}}$ ($i\in D$) 
is multiplicative in a strong form.
\end{enumerate}\end{remark}

\begin{remark}\label{isomorphe-algebren}
Let $\Lambda$ be an unimodular 
association scheme with the associated hypergroup $(D,*)$
as in Proposition \ref{prp-asso-to-hypergroup}. Then,
 by (\ref{rel-stoch-matrices}) and \ref{unimod}(3), 
the (finite) linear span of the matrices
 $S_i$ ($i\in D$) forms an $*$-algebra which is via
$S_i\longleftrightarrow \delta_i$ ($i\in D$)
isomorphic with the $*$-algebra $(M_f(D),*)$ 
of all signed measures on $D$ with finite support and the convolution from 
Proposition \ref{prp-asso-to-hypergroup}.
\end{remark}

We briefly study  concepts of automorphisms of association schemes.

\begin{definition}
 \begin{enumerate}
\item[\rm{(1)}] Let $\Lambda=(X,D=G,(R_{j})_{j\in D})$ and 
$\tilde\Lambda=(\tilde X,
\tilde D,(\tilde R_{j})_{j\in\tilde  D})$ be association schemes.
A pair of $(\phi,\psi)$ of bijective mappings $\phi:X\to\tilde X$,
$\psi:D\to\tilde D$ is called an isomorphism from $\Lambda$
 onto  $\tilde\Lambda$ if for all $x,y\in X$ and $i\in D$ with $(x,y)\in R_i$,
$(\phi(x),\phi(y))\in\tilde R_{\psi(i)}$.
\item[\rm{(2)}] If $\Lambda=\tilde\Lambda$, then an isomorphism is called an
  automorphism.
\item[\rm{(3)}] If $(\phi,\psi)$ is an automorphism with $\psi$ as identity,
  then $\phi$ is called a strong automorphism. The set $Saut(\Lambda)$
of all strong  automorphisms of $\Lambda$ is obviously a group.
\end{enumerate}\end{definition}

\begin{remark}\label{rem-auto1}
 Let $\Lambda=(X,D,(R_{j})_{j\in D})$  be an association
  scheme. If we regard  $Saut(\Lambda)$ as  subspace of the space
 $X^X$ of all maps from $X$ to $X$ with the product topology, then
  $Saut(\Lambda)$ obviously is a topological group. For  $x_0\in X$ we
  consider the stabilizer subgroup 
$Stab_{x_0}:=\{\phi\in Saut(\Lambda):\> \phi(x_0)=x_0\}$ which is obviously
  open in $Saut(\Lambda)$. $Stab_{x_0}$ is also compact. In fact,
for each $g\in Stab_{x_0}$ and $y\in X$ with $(x,y)\in R_i$ with $i\in D$,
we have $(x,g(y))=(g(x),g(y))\in  R_i$. The axioms of an association
scheme show that  $\{g(y):\> g\in  Stab_{x_0}\}$ is finite  $y\in X$.
 This shows  that  $Stab_{x_0}$ is compact. In particular,
 $Stab_{x_0}$ is  a compact open neighborhood of the identity of 
 $Saut(\Lambda)$ which shows that $Saut(\Lambda)$ is locally compact.

This argument is completely analog to  Example \ref{graphauto} for graphs.
\end{remark}

\begin{example}\label{rem-auto}
 Let $G$ be a locally compact group with compact open subgroup
  $H$, or let $(G,H)$ be a Hecke pair. Consider the associated
 association  scheme
 $\Lambda=(X=G/H,D=G//H,(R_{j})_{j\in D})$  as in
 \ref{example-totally-disconnected} or \ref{example-Hecke-condition}.
 Then for each $g\in G$, the mapping $T_g:G/H\to G/H$, $xH\mapsto gxH$ defines a strong
 automorphism of $\Lambda$.
This obviously leads to a group homomorphism $T:G\to Saut(\Lambda)$.
\end{example}

\begin{remark}
Let $(\phi,\psi)$ be an automorphism of an association
  scheme $\Lambda=(X,D,(R_{j})_{j\in D})$. Then $\psi:D\to D$ is an
  automorphism of the associated hypergroup on $D$ according to
Proposition \ref{prp-asso-to-hypergroup}. This follows immediately 
from the definition of
the convolution in \ref{prp-asso-to-hypergroup} and the definitions of
$p_{ij}^k$ and $\omega_i$.\end{remark}

The constructions of
association schemes  are   parallel above for groups $G$
with
compact open subgroups $H$ and for Hecke pairs $(G,H)$.
This is not an accident:

If $\tilde G$ is a locally compact group
with  compact open subgroup
$\tilde H$, then w.r.t.~the discrete topology, $(\tilde G,\tilde H)$ is also a
Hecke pair, and the association scheme $(\tilde X=\tilde G/\tilde H,
\tilde D=\tilde G//\tilde H,(\tilde R_{j})_{j\in\tilde  D})$ does not depend
on the topologies. In this way, the approach via Hecke pairs seems to be 
the more general one. On the other hand, the following theorem shows that
both approaches are equivalent:

\begin{theorem}
 Let $(G,H)$ be a Hecke pair. Then there is a totally
  disconnected locally compact group $\bar G$  with an compact open subgroup
$\bar H$ such that the associated association schemes
 $(G/H,D=G//H,(R_{j})_{j\in D})$ and $(\bar G/\bar H,
\bar D=\bar G//\bar H,(\bar R_{j})_{j\in\bar  D})$ are isomorphic.
The associated  double coset hypergroups are also isomorphic.
\end{theorem}

\begin{proof}
Let
$(G,H)$ be a Hecke pair and $\Lambda:=(G/H,D=G//H,(R_{j})_{j\in D})$
 the associated association scheme. Consider the homomorphism
$T:G\to Saut(\Lambda)$
according to  
\ref{rem-auto}, where  $G$ acts transitively on $X=G/H$.
Hence, the totally disconnected locally compact group 
$\tilde G:=Saut(\Lambda)$ of \ref{rem-auto1} acts transitively on $X=G/H$.
We define $\tilde H:=Stab_{eH}$ as a compact open subgroup of $\tilde G$.

We next consider the normal subgroup $N:=\bigcap_{x\in G} xHx^{-1}\le H$ of $G$.
Then obviously $(G/N)/(H/N)\simeq G/H$, $(G/N)//(H/N)\simeq G//H$,
and $(G/N, H/N)$ forms a Hecke pair for which the associated 
 associated association scheme is isomorphic with $\Lambda$. Using
 this division
 by $N$, we may assume from now on 
that $N=\{e\}$.
 
If this is the case, we obtain that the homomorphism $G\to \tilde G,
g\mapsto T_g$ from \ref{rem-auto} with  $T_g(xH)=gxH$ is injective.
In fact, if for some $g\in G$ and all $x\in G$ we have 
$T_g(xH)=gxH=xH$, it follows
that $g\in N=\{e\}$ as claimed.
We thus may assume that $G$ is a subgroup of $\tilde G$.
We then readily obtain $H=G\cap \tilde H$. 

We now consider the closures $\bar G,\bar H$ of $G,H$ in $\tilde G$.
Then $\bar G$ is a locally compact, totally disconnected topological group
with $\bar H\subset \tilde H$ as compact subgroup. Moreover
$H=G\cap \tilde H$ yields  $\bar H=\bar G\cap \tilde H$ which implies
that  $\bar H$ is in fact a compact open subgroup of $\bar G$.
Moreover, as $\bar G$ acts transitively on $G/H$ with $\bar H$ as stabilizer
subgroup of $eH\in G/H$, we may identify $\bar G/\bar H$ with $G/H$.
To make this more explicit we claim that
\begin{equation}\label{gleich-kl}
\text{for all}\quad \bar g\in\bar G
\quad\text{there exists}\quad g\in G \quad\text{with}
\quad\bar g \bar H= g \bar H.
\end{equation}
In fact, each $\bar g\in\bar G$ is the limit of some $g_\alpha\in G$. 
As $\bar g\bar H$ is an open neighborhood of  $\bar g\in\bar G$, we
obtain   $g_\alpha\in\bar g\bar H$ for large indices which proves 
(\ref{gleich-kl}). Moreover, as for $g_1,g_2\in G$ the relation
$g_1\bar H=g_2\bar H$ implies $g_1^{-1}g_2\in \bar H\cap G=H$, we obtain with
(\ref{gleich-kl}) that
the mapping
$$\phi:G/K\to \bar G/\bar H, \quad gH\mapsto g\bar H$$
is a well defined bijective mapping.
 Moreover,  it is clear that the orbits of the action of $H$ and $\bar H$
on  $G/H$ are equal.
 This shows that also the mapping
$$\psi:G//K\to \bar G//\bar H, \quad HgH\mapsto\bar H g\bar H$$
is  well defined and bijective. If we compare the pair $(\phi,\psi)$
with the construction of the associations schemes in Examples 
\ref{example-totally-disconnected} and 
\ref{example-Hecke-condition}, we see that the  schemes
$\Lambda$ and $(\bar X=\bar G/\bar H,
\bar D=\bar G//\bar H,(\bar R_{j})_{j\in\bar  D})$ are 
isomorphic via $(\phi,\psi)$ as claimed.

Finally, the statement about the associated double coset hypergroups is
clear by 
Example \ref{equal-const}.
\end{proof}

We finally present the following trivial commutativity criterion:

\begin{lemma}\label{crit-comm} If an association scheme  $\Lambda=(X,D,(R_i)_{i\in D})$ admits an 
automorphism $(\phi,\psi)$ with $\psi(i)=\bar i$ for $i\in D$, then 
 $\Lambda$ is commutative.
\end{lemma}

\begin{proof}  For $i,j,k\in D$, $p_{i,j}^k = p_{\bar j, \bar i}^{\bar k}
= p_{\psi(j),\psi(i)}^{\psi(k)} = p_{j,i}^k$ where the last equality follows
easily from the definition of an automorphism and the definition of $p_{j,i}^k$.
\end{proof}

\begin{remark}
 Lemma \ref{crit-comm}  is a natural extension of
  the following well-known criterion for Gelfand pairs:\hfil\break
Let $H$ be a compact subgroup of a locally compact group $G$ such that there
is a continuous  automorphism $T$ of $G$ with $x^{-1}\in HT(x)H$ for $x\in G$.
 Then $(G,H)$ is a Gelfand pair, i.e., $G//H$ is a commutative hypergroup.
\hfil\break
If $H$ is compact and open in $G$, then we obtain this criterion from
 Lemma \ref{crit-comm} applied to 
association scheme  $\Lambda=(X=G/H, D=G//H, (R_i)_{i\in D})$
with the automorphism  $(\phi, \psi)$ with 
$\phi(gH):= T(g)H$ and $\psi(HgH):= HT(g)H$. 
\end{remark}

\section{Commutative association schemes and dual product formulas}

In this section let $\Lambda=(X,D,(R_i)_{i\in D})$ be a commutative
association scheme and $(D,*)$ the associated commutative discrete 
hypergroup as in Proposition \ref{prp-asso-to-hypergroup}.

By Remark \ref{isomorphe-algebren}, the  linear span $A(X)$ 
of the matrices
 $S_i$ ($i\in D$) is a commutative $*$-algebra  which is 
isomorphic with the   $*$-algebra $(M_f(D),*)$ via
$$\mu\in M_f(D) 
\longleftrightarrow S_\mu:=\sum_{i\in D} \mu(\{i\}) S_i\in A(X).
$$
Moreover, as $(f*g)\Omega=f\Omega *g\Omega$ and $f^*\Omega=(f\Omega)^*$
for all $f,g\in C_c(D)$ and the Haar measure $\Omega$ of $(D,*)$ by
\cite{J}, the 
$*$-algebra $A(X)$ is also isomorphic with the commutative  $*$-algebra
$(C_c(D),*)$ via
$$ f\in C_c(D)  \longleftrightarrow S_f:=\sum_{i\in D} f(i)\omega_i S_i=
\sum_{i\in D} f(i) A_i \in A(X).$$
By this construction, $S_f=S_{f\Omega}$ for $ f\in C_c(D)$. 
We notice that even for $f\in C_b(D)$, the matrix
$ S_f:=\sum_{i\in D} f(i)\omega_i S_i$ is well-defined, and that products of 
$ S_f$ with matrices with only finitely many nonzero entries in all rows and
columns are well-defined. In particular, we may multiply $ S_f$ on both sides
with any $S_\mu$, $\mu\in M_f(D)$.

For each function $f\in C_b(D)$ we define the function 
$F_f:X\times X\to\mathbb C$ with $F_f(x,y)=f(i)$ for the unique $i\in D$ with
$(x,y)\in R_i$. Then $F_f$ is constant on all partitions $R_i$, and by our
construction, 
 $$(F_f(x,y))_{x,y\in X} = S_f.$$
In particular,  $f(e)$ is equal to the diagonal entries of $S_f$ (where these are equal).

The following definition is standard for kernels.

\begin{definition}
A function $F:X\times X\to \mathbb C$ is 
called positive definite on $X\times X$
 if for all $n\in\mathbb N$, $x_1,\ldots, x_n\in X$ and
  $c_1,\ldots,c_n\in\mathbb C$,
$$\sum_{k,l=1}^n c_k \bar c_l \cdot F(x_k,x_l)\ge0,$$
i.e., the matrices $( F(x_k,x_l))_{k,l}$ are positive semidefinite.
This in particular implies that these matrices are hermitian, i.e., 
$$ F(x,y)=\overline{F(y,x)} \quad\text{for all}\quad x,y\in X.$$
\end{definition}

As the pointwise products of positive semidefinite matrices are again positive 
 semidefinite (see e.g.~Lemma 3.2 of \cite{BF}), we have:

\begin{lemma}\label{prod-pos}
If  $F,G:X\times X\to\mathbb C$ are positive definite, then the pointwise product
$F\cdot G:X\times X\to\mathbb C$ is also positive definite.
\end{lemma}

Moreover, for $f\in C_b(D)$, the positive definiteness of $f$ on $D$ is
related to that of $F_f$ on $X\times X$. We begin with the following observation:

\begin{lemma}\label{pd-1}
Let  $f\in C_b(D)$. If $F_f$ is positive definite, 
then $f$ is positive definite on the hypergroup $(D,*)$.
\end{lemma}

\begin{proof} For  $f$ define the reflected function $f^-(x) :=f(\bar x)$ for
  $x\in D$.
Let $n\in\mathbb N$, $x_1,\ldots, x_n\in D$ and
  $c_1,\ldots,c_n\in\mathbb C$. Let $\mu:=\sum_{k=1}^n c_k\delta_{x_k}\in M_f(D)$
  and observe that by the definition of the convolution of a
 function with a measure and by commutativity,
$$P:=\sum_{k,l=1}^n c_k \bar c_l \cdot f^-(x_k*\bar x_l)
 =\int_D  f^-\> d(\mu*\mu^*)=(\mu*\mu^**f)(e)= (\mu*f*\mu^*)(e).$$
where by the considerations above, this is equal to the diagonal entries of 
$$S_{\mu*f*\mu^*}= S_\mu \cdot S_{f}\cdot \overline{S_\mu}^T=
S_\mu\cdot F_f(x,y)_{x,y\in X}\cdot\overline{S_\mu}^T .$$
This matrix product exists and is positive semidefinite, as so is
$F_f(y,x)_{x,y\in X}$ by our assumption. As the diagonal entries of a positive 
semidefinite matrix are nonnegative, we obtain  $P\ge 0$.
Hence,  $f^-$ and thus $f$ is positive definite. 
\end{proof}

 Here is a partially reverse statement.

\begin{lemma}\label{pd-2}
Let  $f\in C_c(D)$. Then, by \ref{facts-hy}(2),  
$f*f^*$ is positive definite on $D$, and
 $F_{f*f^*}$ is positive definite.
\end{lemma}

\begin{proof} By our considerations above,
$(F_{f*f^*}(x,y))_{x,y\in X} = S_{f*f^*}=S_f\overline{S_f}^T.$
This proves that  $F_{f*f^*}$ is positive definite.
\end{proof}

We now combine Lemma \ref{pd-2} with Fact \ref{facts-hy}(4) and the trivial
observation that pointwise limits of positive definite functions are 
positive definite. This shows:

\begin{corollary}\label{pd3}
Let $\alpha\in S\subset\hat D$ be a character in the support
 of the Plancherel measure. Then $F_\alpha:X\times X\to\mathbb C$ 
is positive definite.
\end{corollary}

Lemmas \ref{prod-pos} and \ref{pd-1} now lead to the
following central result of this paper:

\begin{theorem}\label{main}
 Let $(D,*)$ be a commutative discrete hypergroup which is associated with
 some association scheme $\Lambda=(X,D,(R_i)_{i\in D})$.
Let $\alpha,\beta\in S\subset\hat D$ be  characters in the support
 of the Plancherel measure. Then $\alpha\cdot\beta$ 
is  positive definite  on $D$, and 
there exists a unique probability 
measure $\delta_\alpha\hat*\delta_\beta\in M^1(\hat D)$ with 
$(\delta_\alpha\hat*\delta_\beta)^\vee=\alpha\cdot\beta$.
The support of
this measure is contained in $S$.

Furthermore, for all $\alpha\in S$, 
the unique positive character $\alpha_0$ in $S$ according to 
 \ref{facts-hy}(5)
is contained in the support of
$\delta_\alpha\hat*\delta_{\bar\alpha}$.
\end{theorem}

\begin{proof}
\ref {pd3}, \ref{prod-pos}, and \ref{pd-1} show
that  $\alpha\cdot\beta$ 
is  positive definite. Bochner's theorem  \ref{facts-hy}(1) now leads
to the probability measure $\delta_\alpha\hat*\delta_\beta\in M^1(\hat D)$.
The assertions about the supports of $\delta_\alpha\hat*\delta_\beta$ and
$\delta_\alpha\hat*\delta_{\bar\alpha}$ follow from Theorem 2.1(4) of \cite{V3}.
\end{proof}

For finite association schemes we  have a stronger
result. It is shown in Section 2.10 of \cite{BI} in another, but finally
equivalent way:

\begin{theorem}\label{main1} 
Let $(D,*)$ be a finite commutative  hypergroup which is associated with
 some association scheme $\Lambda=(X,D,(R_i)_{i\in D})$. Then $(\hat D, \hat *)$
is a hypergroup.
\end{theorem}

\begin{proof} For finite hypergroups we have $S=\hat D$, and the unique
positive character in $S$ is the identity ${\bf 1}$. Therefore, 
if we take ${\bf 1}$ as identity and 
complex conjugation as involution, we see that almost all hypergroup
properties  of $(\hat D,\hat *)$ follow from Theorem \ref{main}.
We only have to check that for $\alpha\ne\beta\in \hat D$,
 ${\bf 1}$ is not contained in the support of 
$\delta_\alpha\hat *\delta_{\bar\beta}$. For this we recapitulate that
 for all $\gamma,\rho \in\hat D$, 
$\hat\gamma(\rho)=\int\gamma\bar\rho\>d\Omega=\|\gamma\|_2^2 \delta_{\gamma,\rho}$ with the Kronecker-$\delta$.
Therefore, with 12.16 of \cite{J},
\begin{align} (\delta_\alpha\hat *\delta_{\bar\beta})(\{{\bf 1}\})&=
\int_{\hat D}  {\bf 1}_{\{{\bf 1}\}}\> d(\delta_\alpha\hat *\delta_{\bar\beta})= 
\int_{\hat D} \hat{\bf 1}\> d(\delta_\alpha\hat *\delta_{\bar\beta}) =\notag\\
&=\int_{ D} {\bf 1}\> (\delta_\alpha\hat *\delta_{\bar\beta})^\vee\> d\Omega = \int_{ D} \alpha\bar\beta \> d\Omega=
\|\gamma\|_2^2 \delta_{\alpha,\beta}=0.
\notag\end{align}
This completes the proof.
\end{proof}

We present some examples related to Gelfand pairs which show that the
 infinite case is more involved than in Theorem \ref{main1}; 
for details see \cite{V4}.

\begin{example}\label{ex-dist}
 Let $a,b\ge2$ be integers. Let  $C_b$ the complete undirected graph
graph with $b$
 vertices, i.e.,  all vertices of $C_b$ are connected.
We now consider the infinite
   graph $\Gamma:=\Gamma(a,b)$  where
 precisely $a$ copies of the graph $C_b$ are tacked together at each vertex 
in a tree-like way, i.e., there are no other cycles in  $\Gamma$ than
those in a  copy  of $C_b$. For $b=2$, $\Gamma$ is the homogeneous tree of
valency $a$. We denote the distance function on $\Gamma$ by $d$.

It is  clear
 that the group $G:=Aut(\Gamma)$
of all graph automorphisms acts on $\Gamma$ in a  distance-transitive way, i.e.,
for all $v_1,v_2,v_3,v_4\in\Gamma$ with
$d(v_1,v_3)=d(v_3,v_4)$ there exists  $g\in\Gamma$
 with $g(v_1)=v_3$ and $g(v_2)=v_4$.
 $Aut(\Gamma)$ is a totally disconnected, locally compact group w.r.t.~the
 topology of pointwise convergence, and
 the stabilizer subgroup $H\subset G$ of any fixed vertex
$e\in\Gamma$ is  compact and open. We  identify $G/H$ with
$\Gamma$, and $G//H$ with $\mathbb N_0$  by distance transitivity.
We now study the association scheme
  $\Lambda=(\Gamma\simeq G/H,\mathbb N_0=G//H, (R_i)_{i\in\mathbb N_0}) $
and the  double coset hypergroup $(\mathbb N_0\simeq G//H, *)$.
As in the case of finite  distance-transitive graphs in \cite{BI}, 
 $\Lambda$ and  $(\mathbb N_0, *)$ are commutative and associated with a
sequence of orthogonal polynomials in the Askey scheme \cite{AW} .

More precisely, by \cite {V1a}, the hypergroup convolution is given by
\begin{equation}\label{faltung}
\delta_m*\delta_n = \sum_{k=|m-n|}^{m+n} g_{m,n,k} \delta_k\in  M^1(\mathbb
N_0)
\quad\quad  (m,n\in\mathbb N_0)
\end{equation}
with
$$ g_{m,n, m+n}= \frac{a-1}{a}>0, \quad
g_{m,n, |m-n|}=  \frac{1}{a(a-1)^{m\wedge n-1} (b-1)^{m\wedge n}}>0,$$
$$g_{m,n,|m-n|+2k+1}= \frac{b-2}{ a(a-1)^{m\wedge n-k-1}(b-1)^{m\wedge n-k}}\ge0
\quad  (k=0,\ldots ,m\wedge n-1), $$
$$g_{m,n,|m-n|+2k+2}= \frac{a-2}{(a-1)^{m\wedge n-k-1}(b-1)^{m\wedge n-k-1}}\ge0
\quad (k=0,\ldots,m\wedge n-2).$$
The Haar weights are  given by $\omega_0:=1$, 
$\omega_n=a(a-1)^{n-1}(b-1)^n \quad (n\ge1)$.
Using
$$g_{n,1,n+1}=\frac{a-1}{a }, \quad 
g_{n,1,n}=  \frac{b-2}{a(b-1)}, \quad
 g_{n,1,n-1}=\frac{1}{a(b-1)},$$
we define a sequence of orthogonal polynomials
$(P_n^{(a,b)})_{n\ge0}$ by  
 $$P_0^{(a,b)}:=1, \quad\quad 
P_1^{(a,b)}(x):= \frac{2}{a}\cdot\sqrt{\frac{a-1}{b-1}}\cdot x +
\frac{b-2}{a(b-1)},$$
and the three-term-recurrence relation
\begin{equation}\label{recu}
P_1^{(a,b)}P_n^{(a,b)}= \frac{1}{a(b-1)}P_{n-1}^{(a,b)}
 + \frac{b-2}{a(b-1)} P_n^{(a,b)} +
\frac{a-1}{a }P_{n+1}^{(a,b)} \quad\quad(n\ge1) .  
\end{equation}
 Then, 
\begin{equation}\label{prodcartier}\textstyle
 P_m^{(a,b)}P_n^{(a,b)}= \sum_{k=|m-n|}^{m+n} g_{m,n,k}P_k^{(a,b)} \quad (m,n\ge0). 
\end{equation}

We also notice that  the  formulas above are correct for all 
 $a,b\in[2,\infty[$, 
and that
Eq.~(\ref{faltung}) then still  defines  commutative  hypergroups 
 $(\mathbb N_0,*)$.

We discuss some  properties of the  $P_n^{(a,b)}$ from 
\cite{V1a}, \cite{V4}. 
Eq.~(\ref{recu}) yields
\begin{equation}\label{potenz}
P_n^{(a,b)}\bigl(\frac{z+z^{-1}}{2}\bigr)= \frac{c(z)z^n +c(z^{-1})z^{-n}}
{((a-1)(b-1))^{n/2}} \quad\quad\text{for}\>\> z\in\mathbb C\setminus\{0, \pm 1\}
\end{equation}
with
\begin{equation}\textstyle
c(z):=\frac{(a-1)z -z^{-1} +(b-2)(a-1)^{1/2}(b-1)^{-1/2}}
{a(z-z^{-1})}.
\end{equation}
We  define
\begin{equation}\label{xnull}
 s_0:=  s_0^{(a,b)}:= \frac{2-a-b}{ 2\sqrt{(a-1)(b-1)}}, 
\quad\quad
s_1:= s_1^{(a,b)}:=\frac{ ab -a-b+2}{ 2\sqrt{(a-1)(b-1)}}.
\end{equation}
Then
\begin{equation}\label{reinerpofall}
P_n^{(a,b)}\left(s_1\right)=1,\quad\quad
 P_n^{(a,b)}\left(s_0 \right)=(1-b)^{-n} \quad\quad(n\ge0).
\end{equation}
It is shown in \cite{V4}
 that the
 $P_n^{(a,b)}$ fit into the Askey-Wilson
 scheme (pp.~26--28 of 
 \cite{AW}).
 By the orthogonality relations in \cite{AW}, the 
normalized orthogonality measure $\rho=\rho^{(a,b)}\in M^1(\mathbb R)$
 is 
\begin{equation}\label{orthohne}
d\rho^{(a,b)}(x)= w^{(a,b)}(x)dx\Bigr|_{[-1,1]} 
\quad\quad{\rm for}\quad a\ge b\ge 2
\end{equation}
and 
\begin{equation}\label{orthmit}
d\rho^{(a,b)}(x)=w^{(a,b)}(x)
dx\Bigr|_{[-1,1]} + \frac{b-a}{b} d\delta_{s_0}
  \quad{\rm for}\quad b> a\ge 2
\end{equation}
with
$$w^{(a,b)}(x):=\frac{a}{2\pi} \cdot  \frac{(1-x^2)^{1/2}}{(s_1-x)(x-s_0)}.$$
For $a,b\in\mathbb R$ with $a,b\ge 2$, the numbers $s_0,s_1$ satisfy 
$$-s_1\le s_0\le -1 < 1\le s_1.$$
By
Eq.~(\ref{potenz}), we have the 
dual space
 $$\hat D \simeq
\{x\in\mathbb R:\> (P_n^{(a,b)}(x))_{n\ge0} \quad{\rm is \>\>
 bounded\>} \}= [-s_1, s_1].$$
This interval  contains the support $S:=supp\> \rho^{(a,b)}$ 
of the orthogonality measure, which is also equal to the Plancherel measure.
We have $S=\hat D $ precisely for $a=b=2$.
The following theorem from \cite{V4} shows that for these examples several
 phenomena appear, and that  Theorem \ref{main} 
cannot be improved considerably.
\end{example}

\begin{theorem}\label{podi} Let $a,b\ge 2$ be integers.
Let $G,H,\Gamma$, and $\Lambda$  be given as above. Then the following
 statements are equivalent for $x\in\mathbb R$:
\begin{enumerate}
\item[\rm{(1)}]  $x\in [s_0^{(a, b)}, s_1^{(a, b)}]$.
\item[\rm{(2)}] The mapping 
$\Gamma\times\Gamma\to\mathbb R$, 
 $(v_1,v_2)\longmapsto P_{d(v_1,v_2)}^{( a, b)}(x)$ is 
 positive definite;
\item[\rm{(3)}] The mapping $g \longmapsto  P_{d(gH,e)}^{( a, b)}(x)$
is  positive definite on $G$.
\end{enumerate}
Moreover, for all $x,y\in[s_0^{(a, b)}, s_1^{(a, b)}]$
 there exists a unique probability measure
 $\mu_{x,y}\in M^1([-  s_1^{(a, b)},  s_1^{(a, b)}])$ with
\begin{equation}\label{special-dist-prof}
P_n^{(a, b)}(x)\cdot P_n^{(a, b)}(y)= \int_{- s_1^{(a, b)}}^{s_1^{(a, b)}}
P_n^{(a, b)}(z )\> d\mu_{x,y}(z) \quad { for\>\> all}
\quad n\in\mathbb N_0.\end{equation}

Furthermore, for certain integers $b>a$ there exist 
$x,y\in[-s_1^{(a, b)},s_0^{(a, b)}[$ for which no probability measure 
$\mu_{x,y}\in M^1(\mathbb R)$ exists which satisfies 
(\ref{special-dist-prof}).
\end{theorem}

\begin{remark} The preceding examples show that for Gelfand pairs $(G,H)$ with
  $H$ a compact open subgroup, characters $\alpha\in (G//H)^\wedge$ of the
  double coset hypergroup may not correspond to a positive definite function
  on $G$ and not to a positive definite kernel on $G/H$. On the other hand, as
  in the examples in parts (2) and (3) of Theorem \ref{podi}, these positive
  definiteness conditions are equivalent in general. This well-known fact 
follows immediately
  from the definitions of positive definite functions on $G$ and 
positive definite kernel on $G/H$. 
\end{remark}

\section{Generalized  discrete association schemes}

In this section we propose a generalization
of  the definition of association schemes from Section 3 via
stochastic matrices. We focus on the commutative case and have
applications in mind to  dual product formulas as in Theorem
\ref{main}.

One might expect at a first glance 
that for each discrete commutative hypergroup $(D,*)$ there is a
natural kind of a generalized 
 association scheme with $X=D$ and transition matrices $S_i$
with $(S_i)_{j,k}:=(\delta_i*\delta_j)(\{k\})$ for $i,j,k\in D$. 
We then would obtain $F_f(x,y)= f(x*y)$  for $x,y\in D$ and $f\in C(D)$ 
with the notations in Section 4. These 
matrices $S_i$ and the 
functions  $F_f$ have much in common with the construction 
above. However, one central point is different in the view of dual product formulas:
In Section 4 we used the fact
 that the $R_i$ form  partitions  which led to the central identity
$F_{fg}=F_f F_g$  for functions $f,g$ on $D$. This is obviously usually 
not correct for $F_f(x,y)= f(x*y)$. In particular, such
a simple approach would lead to  a contradiction with Example 9.1C of Jewett
\cite{J}.

Having this  problem in mind, 
we propose a  definition which skips the
 integrality condition and which keeps
the  partition condition. The  integrality conditions will be replaced
by the existence of a positive invariant measure on $X$ which
replaces the counting measure on $X$ in  Sections 3 and 4.

\begin{definition}\label{def-gen-discrete-asso}
 Let $X, D$ be nonempty, at most countable 
sets and $(R_i)_{i\in D}$ a disjoint
  partition of $X\times X$ with $R_i\ne\emptyset$ for $i\in D$. 
Let $\tilde S_i\in \mathbb R^{X\times X}$ for $i\in D$ be stochastic
matrices. 
Assume that:
\begin{enumerate}
\item[\rm{(1)}] For all $i,j,k\in D$ and $(x,y)\in R_k$, the number
$$p_{i,j}^k:=|\{z\in X:\> (x,z)\in R_i\>\> {\rm and }\>\> (z,y)\in R_j\}|$$
is finite and independent of $(x,y)\in R_k$.
\item[\rm{(2)}] For all $i\in D$ and $x,y\in X$, $\tilde S_i(x,y)>0$ 
if and only if $(x,y)\in R_i$.
\item[\rm{(3)}] For all $i,j,k\in D$ there exist (necessarily nonnegative)  numbers
$\tilde p_{i,j}^k$ with $\tilde S_i\tilde S_j=\sum_{k\in D} \tilde p_{i,j}^k\tilde S_k$.
\item[\rm{(4)}] There exists an identity $e\in D$ with $\tilde S_e=I_X$ as identity matrix.
\item[\rm{(5)}] There exists a positive measure $\pi\in M^+(X)$ with
  $supp\>\pi =X$ and an involution $i\mapsto \bar i$ on $D$ 
such that
for all $i\in D$, $x,y\in X$,
  $$\pi(\{y\})\tilde  S_{\bar i}(y,x)=\pi(\{x\})\tilde   S_{ i}(x,y).$$
\end{enumerate}
Then $\Lambda:=(X,D, (R_i)_{i\in D},(\tilde  S_i)_{i\in D})$ 
is called a generalized association scheme. 

 $\Lambda$ is called commutative if $\tilde S_i\tilde S_j=\tilde S_j\tilde S_i$
 for all $i,j\in D$. It is called symmetric  if
the involution  is the identity. $\Lambda$ is called finite, if so are $X$ and  $D$.
\end{definition}

\begin{example}
If  $\Lambda=(X,D, (R_i)_{i\in D})$ 
is an unimodular association scheme as in Section 3, and if we take the 
stochastic matrices $ (S_i)_{i\in D}$ as in  Remark \ref{remark-asso-stoch},
 then 
$(X,D, (R_i)_{i\in D}, (S_i)_{i\in D})$  is a generalized association
scheme. In fact, axioms (1)-(4) are clear, and for axiom (5) we take the
involution of $\Lambda$ and $\pi$ as the counting measure on $X$. 
Lemma  \ref{unimod}(3)
and  unimodularity then imply axiom (5).
\end{example}

\begin{lemma}\label{basic-gen}
Let $(X,D, (R_i)_{i\in D}, (\tilde S_i)_{i\in D})$ be a generalized
  association scheme. Then:
\begin{enumerate}
\item[\rm{(1)}] The triplet $(X,D, (R_i)_{i\in D})$ 
is an association scheme. 
\item[\rm{(2)}] The positive measure $\pi\in M^+(X)$ from axiom (5) 
is invariant,
i.e.,
for all $y\in X$ and $i\in D$, 
$\sum_{x\in X} \pi(\{x\})\tilde  S_i(x,y)=  \pi(\{y\}).$
\item[\rm{(3)}] The deformed intersection numbers
$\tilde p_{i,j}^k$ 
  satisfy
 $\sum_{k\in D}\tilde p_{i,j}^k=1$ and
$$\tilde p_{i,j}^k>0 \quad \Longleftrightarrow 
\quad p_{i,j}^k>0\quad \quad (i,j,k\in D).  $$
\item[\rm{(4)}] For all $p\in[1,\infty[$ and $i\in D$,
the transition operator 
$$\textstyle T_{\tilde S_i}f(x):=\sum_{y\in X} \tilde S_i(x,y)f(y)
\quad (x\in X)$$
 is
 a continuous linear operator on $L^p(X,\pi)$ with 
$\|T_{\tilde S_i}\|\le1$. Moreover, for $p=2$, the $T_{\tilde S_i}$
satisfy the adjoint relation 
 $T_{\tilde S_{\bar i}}=T_{\tilde S_{ i}}^*$.
\end{enumerate}
\end{lemma}

\begin{proof} Part (1) is clear; we only remark that the axiom regarding the
  involution on $D$ follows from axioms (2) and (5) above.

For part (2) we use axiom (5) which implies for $i\in D$, $y\in X$ that
$$\textstyle
\sum_{x\in X} \pi(\{x\}) S_i(x,y)=\pi(\{x\})\sum_{x\in X} S_i(y,x)=\pi(\{y\}).$$
Part (3) and the adjoint relation in (4) are obvious. Moreover,
for $p\ge1$, the invariance of $\pi$ w.r.t.~$S_i$ and H\"older's inequality 
easily imply
$\|T_{\tilde S_i}\|\le1$.
\end{proof}

In summary, we may regard generalized association schemes as deformations
of given classical association schemes with deformed intersection numbers 
$\tilde p_{i,j}^k$.

Generalized association schemes also lead to discrete
hypergroups:

\begin{proposition}\label{prp-gen-asso-to-hypergroup}
 Let $\Lambda:=(X, D, (R_i)_{i\in D}, (\tilde S_i)_{i\in D})$ be a generalized association
  scheme with deformed intersection numbers $\tilde p_{i,j}^k$.
Then the product $\tilde*$ with
$$\textstyle\delta_i\tilde*\delta_j:=\sum_{k\in D}
 \tilde p_{i,j}^k \delta_k$$
can be extended uniquely to an associative, bilinear, 
$\|\>.\>\|_{TV}$-continuous mapping on $M_b(D)$.
 $(D,\tilde*)$  is a discrete hypergroup with the left and right Haar measure
$$\textstyle\Omega_l:=\sum_{i\in D} \omega_i \delta_i
\quad\text{and}\quad\Omega_r:=
\sum_{i\in D} \omega_{\bar i} \delta_i \quad\text{with}\quad
\omega_i:=\frac{1}{\tilde p_{i,\bar i}^e}>0  $$
 respectively. This  hypergroup is commutative or  
 symmetric  if and only if so is $\Lambda$.
\end{proposition}

\begin{proof} As the product of matrices is associative,  the
  convolution $\tilde*$ is associative.
 Moreover, by Lemma \ref{basic-gen}(3), $\delta_i\tilde*\delta_j$ is a
probability measure for $i,j\in D$ whose support
is the same as for the
hypergroup convolution of the association scheme $(X, D, (R_i)_{i\in D})$
in Proposition \ref{prp-asso-to-hypergroup}. This readily shows that 
$(D,\tilde*)$
is a hypergroup.

Clearly, $\tilde*$ is commutative or symmetric if and only if so is 
 $\Lambda$. Finally, the statement about the Haar measures follows from
Example \ref{example-haar}(1).
\end{proof}

We now  extend the approach of Section 4 to dual positive product
formulas for discrete commutative hypergroups $(D,\tilde*)$ which are associated with 
generalized commutative association schemes
 $\Lambda:=(X, D, (R_i)_{i\in D}, (\tilde S_i)_{i\in D})$.

For this we study the linear span
$A(X)$ 
of the matrices $\tilde S_i$ ($i\in D$). If we identify the $\tilde S_i$
with the transition operators $T_{\tilde  S_i}$, we may regard $A(X)$ as a 
commutative $*$-subalgebra of the $C^*$-algebra 
${\mathcal B}(L^2(X,\pi))$     of all bounded linear operators on $L^2(X,\pi)$
by Lemma \ref{basic-gen}(4). 
 As in Section 4, 
$A(X)$  is 
isomorphic with the  $*$-algebra $(M_f(D),\tilde*)$ of all measures with finite
support via
$$\textstyle \mu\in M_f(D) 
\longleftrightarrow \tilde S_\mu:=\sum_{i\in D} \mu(\{i\}) \tilde S_i\in A(X).
$$
Moreover, using the Haar measure $\Omega$ of $(D,\tilde*)$, 
 $A(X)$ is also isomorphic with the commutative  $*$-algebra
$(C_c(D),\tilde*)$ via
$$\textstyle f\in C_c(D)  \longleftrightarrow \tilde S_f:=\sum_{i\in D} f(i)\omega_i
\tilde 
S_i\in A(X).$$
As in Section 4, we have $\tilde S_f=\tilde S_{f\Omega}$ for $ f\in C_c(D)$. 

We even may define the matrices 
$\tilde S_f:=\sum_{i\in D} f(i)\omega_i\tilde S_i$ 
for $f\in C_b(D)$, and we may form matrix products of
$\tilde S_f$ with matrices with only finitely many nonzero entries 
in all rows and
columns.
In particular, we  
may multiply $\tilde S_f$ on both sides
with any $\tilde S_\mu$, $\mu\in M_f(D)$.
We need the following notion of positive definiteness:

\begin{definition}\label{pi-pd}
Let $A\in\mathbb C^{X\times X}$ be a matrix. Then for all $g_1,g_2\in
C_f(X):=\{g:X\to\mathbb C \quad\text{with finite support}\}$, we may form
$$\textstyle \langle Ag_1,g_2\rangle_\pi:= \sum_{x\in X} (Ag_1)(x)\cdot
\overline{g_2(x)}
\cdot \pi(\{x\})\in \mathbb C.$$
We say that $A$ is $\pi$-positive definite, if $\langle Ag,g\rangle_\pi\ge0$
for all $g\in C_f(X)$. This is obviously equivalent to the fact that the
matrix $(\pi(\{x\})\cdot A_{x,y})_{x,y\in X}$ is positive semidefinite in the
usual way.
\end{definition}

For each $f\in C_b(D)$ we define the function 
$F_f:X\times X\to\mathbb C$ with $F_f(x,y)=f(i)$ for the unique $i\in D$ with
$(x,y)\in R_i$ as in Section 4. Then  different from Section 4,
we usually have 
 $(F_f(x,y))_{x,y\in X} \ne\tilde  S_f$. We thus have to state 
some results from Section 4 for the matrices $\tilde  S_f$ instead of
 $(F_f(x,y))_{x,y\in X}$.

Lemmas \ref{pd-1} and \ref{pd-2} now read as follows:

\begin{lemma}\label{pd-1g}
Let  $f\in C_b(D)$. If $\tilde  S_f$ is $\pi$-positive definite, 
then $f$ is positive definite on the hypergroup $(D,\tilde*)$.
\end{lemma}

\begin{proof} For  $f$ define  $f^-(x) :=f(\bar x)$ for
  $x\in D$.
Let $n\in\mathbb N$, $x_1,\ldots, x_n\in D$ and
  $c_1,\ldots,c_n\in\mathbb C$. Let $\mu:=\sum_{k=1}^n c_k\delta_{x_k}\in M_f(D)$
  and observe that as in the proof of Lemma \ref{pd-1},
$P:=\sum_{k,l=1}^n c_k \bar c_l \cdot f^-(x_k\tilde*\bar x_l)= 
(\mu\tilde*f\tilde*\mu^*)(e)$
where  this is equal to the diagonal entries of 
$$\tilde S_{\mu\tilde*f\tilde*\mu^*}=\tilde S_\mu \cdot\tilde S_{f}\cdot\tilde{S_\mu}^*.$$
This  matrix product exists and is $\pi$-positive semidefinite.
Therefore, by \ref{pi-pd}, the diagonal entries of this matrix are nonnegative
and thus $P\ge 0$.
Hence,  $f^-$ and thus $f$ is positive definite. 
\end{proof}

\begin{lemma}\label{pd-2g}
Let  $f\in C_c(D)$. Then, by \ref{facts-hy}(2),  
$f\tilde*f^*$ is positive definite on $(D,\tilde*)$, and
 $S_{f\tilde*f^*}$ is $\pi$-positive definite.
\end{lemma}

\begin{proof} By our considerations above,
$S_{f\tilde*f^*}=\tilde S_f\tilde{S_f}^*$. This proves the claim.
\end{proof}

We now combine Lemma \ref{pd-2g} with \ref{facts-hy}(4) and conclude as
in Corollary \ref{pd3}:

\begin{corollary}\label{pd3g}
Let $\alpha\in ( D,\tilde*)^\wedge$ be a character in the support 
 of the Plancherel measure. Then $\tilde S_\alpha$ 
is $\pi$-positive definite.
\end{corollary}

We now want to combine Corollary  \ref{pd3g} with
 Lemmas \ref{prod-pos} and \ref{pd-1g} to derive an extension of 
Theorem \ref{main}. Here, however  we would need 
 $(F_f(x,y))_{x,y\in X} =\tilde  S_f$. We can overcome this problem with
some additional condition which relates the characters of 
$(D,\tilde*)$ with the characters of the commutative hypergroup
$(D,*)$ which is associated with the association scheme 
$(X,D,(R_i)_{i\in D})$:

\begin{definition}
Let $\Lambda:=(X, D, (R_i)_{i\in D}, (\tilde S_i)_{i\in D})$ 
be a generalized commutative association scheme with the associated 
hypergroups $(D,\tilde*)$ and $(D,*)$ as  above.
 We say that  $\alpha\in 
(D,\tilde*)^\wedge$ has the positive
connection property if $\alpha$ is positive definite on  $(D,*)$, and if the
associated kernel $F_\alpha:X\times X\to\mathbb C$ is positive semidefinite.

We say that  $\Lambda$ has the positive
connection property, if all $\alpha\in 
(D,\tilde*)^\wedge$ have the positive
connection property.
\end{definition}

To understand this condition, consider a finite
 generalized commutative association scheme $\Lambda$.
If a character $\alpha\in (D,\tilde*)^\wedge$ 
can be written as a nonnegative linear combination of the characters of
 $(D,*)$, then $\alpha$ is positive definite on  $(D,*)$, and by 
Corollary \ref{pd3}, the associated kernel $F_\alpha$ is positive semidefinite.

In this way, the positive connection property of $\Lambda$ 
roughly means that all
 $\alpha\in (D,\tilde*)^\wedge$ admit  nonnegative integral representations
w.r.t.~the characters of  $(D,*)$. Such positive integral representations
are well known for many families of special functions and often easier to prove
than positive product formulas.

We now use the positive
connection property and obtain the following extension of Theorem \ref{main}:

\begin{theorem}\label{main-g}
 Let $(D,\tilde*)$ be a commutative discrete hypergroup which
corresponds to
 the generalized  association scheme 
$\Lambda:=(X,D, (R_i)_{i\in D},(\tilde  S_i)_{i\in D})$.
Let $\alpha,\beta\in (D,\tilde*)^\wedge$ be  characters 
such that $\alpha$ is  in the  support $S$
 of the Plancherel measure of $(D,\tilde*)$ and 
 $\beta$  has the positive
connection property. 
 Then $\alpha\cdot\beta$ 
is  positive definite  on $(D,\tilde*)$, and there is
a unique probability 
measure $\delta_\alpha\hat{\tilde*}\delta_\beta\in M^1((D,\tilde*)^\wedge)$ with 
$(\delta_\alpha\hat{\tilde*}\delta_\beta)^\vee=\alpha\cdot\beta$.
The support of
this measure is contained in $S$.
\end{theorem}

\begin{proof}
By Corollary \ref{pd3g}, $\tilde S_\alpha$ is $\pi$-positive definite, i.e.,
 $ (\pi(\{x\})(\tilde S_\alpha)_{x,y})_{x,y\in X}$ is a positive
semidefinite matrix. Furthermore, as $\beta$ has the positive connection property, 
$(F_\beta(x,y))_{x,y\in X}$ is also positive
semidefinite. We conclude from Lemma
\ref{prod-pos} that the pointwise product 
$ (\pi(\{x\})(\tilde S_\alpha)_{x,y}F_\beta(x,y))_{x,y\in X}$ is also 
positive semidefinite. On the other hand, for all $x,y\in X$ and $i\in D$ with
$(x,y)\in R_i$, we have
$$(\tilde S_\alpha)_{x,y}F_\beta(x,y) =\beta(i)\alpha(i)\omega_i (\tilde S_i)_{x,y}
=(\tilde S_{\alpha\beta})_{x,y}.$$
Therefore, $\tilde S_{\alpha\beta}$ is $\pi$-positive definite, and, by Lemma
\ref{pd-1g}, $\alpha\beta$ is positive definite on $(D,\tilde*)$.
As in the proof of Theorem \ref{main}, Bochner's theorem 
 \ref{facts-hy}(1) together with Theorem 2.1(4) of \cite{V3}
lead to the theorem.
\end{proof}

We remark that Theorem \ref{main1}  can be also established 
for finite generalized commutative association schemes similar to Theorem 
\ref{main-g}.

We complete the paper with an example.

\begin{example}
Consider the group $(\mathbb Z,+)$, on which  
the group $H:=\mathbb Z_2=\{\pm 1\}$ acts multiplicatively
 as group of automorphisms. Consider the semidirect product 
$G:=\mathbb Z  \rtimes \mathbb Z_2$ with $H:=\mathbb Z_2$ as finite subgroup.
We then consider $X:=G/H=\mathbb Z$ and $D:=G//H= \mathbb N_0$ (with the
canonical identifications) as well as the corresponding commutative
association scheme  $(\mathbb Z,\mathbb N_0, (R_k)_{k\in \mathbb N_0})$  with
the associated double coset hypergroup $( \mathbb N_0,*)$ and the associated
  transition matrices
$$S_0=I_{\mathbb Z}, \quad S_k(x,y)=\frac{1}{2}\delta_{k,|x-y|} 
\quad(k\in \mathbb N, \> x,y\in \mathbb Z)$$
with the Kronecker-$\delta$.

Now fix some parameter $r>0$ and put $p:=e^r/(e^r+e^{-r})\in ]0,1[$. Define
    the deformed stochastic matrices
$$\tilde S_0=I_{\mathbb Z}, \quad \tilde S_k(x,y)=\frac{1}{p^k+(1-p)^k}
\Bigl(p^k \delta_{k,y-x} +(1-p)^k \delta_{k,x-y}\Bigr)$$
for $k\in \mathbb N_0, \> x,y\in \mathbb Z$. It can be easily checked that for
$k\in\mathbb N$,
\begin{align}
\tilde S_k\tilde S_1 &= \frac{p^{k+1}+(1-p)^{k+1}}{p^k+(1-p)^k}\tilde S_{k+1}
+\Bigl(1-\frac{p^{k+1}+(1-p)^{k+1}}{p^k+(1-p)^k}\Bigr)\tilde S_{k-1}\notag\\
&=\frac{\ch((k+1)r)}{2\ch(kr)\ch(r)}\tilde S_{k+1}+
\frac{\ch((k-1)r)}{2\ch(kr)\ch(r)}\tilde S_{k-1}.
\end{align}
Induction yields that for $k,l\in\mathbb N_0$,
\begin{equation}\label{falt-ch}
\tilde S_k\tilde S_l=\frac{\ch((k+l)r)}{2\ch(kr)\ch(lr)}\tilde S_{k+l}+
\frac{\ch((k-l)r)}{2\ch(kr)\ch(lr)}\tilde S_{|k-l|}.
\end{equation}
We thus obtain the axioms (1)-(4) of \ref{def-gen-discrete-asso}. 
Moreover, with the measure
$\pi(\{x\}):= \Bigl(\frac{p}{1-p}\Bigr)^x = e^{2rx}$ ($x\in\mathbb Z$) and the
identity as involution on $\mathbb N_0$ we also obtain axiom
 \ref{def-gen-discrete-asso}(5). We conclude from (\ref{falt-ch}) that
the associated hypergroup $(\mathbb N_0,\tilde*)$ is the so called discrete
cosh-hypergroup; see \cite{Z} and  3.4.7 and 3.5.72 of \cite{BH}. The
characters of  $(\mathbb N_0,\tilde*)$ are given by
$$\alpha_{\lambda}(n):=\frac{\cos(\lambda n)}{\ch(rn)} \quad(n\in \mathbb
N_0,\>\>\lambda\in [0,\pi]\cup i\cdot[0,r]\cup \{\pi+iz: \> z\in [0,r]\})$$
where in this parameterization, $\alpha_{\lambda}$ is in the support $S$ of the 
Plancherel measure precisely for $\lambda\in [0,\pi]$. Using
$$\frac{\cos(\lambda n)}{\ch(rn)}=
\frac{1}{2}\int_{-\infty}^\infty \frac{\cos(tr n)}{\ch((t+\lambda/r)\pi/2)}\>
  dt
\quad \text{for}\quad \lambda\in\mathbb C, \> |\Im \lambda|<r$$
(see (1) in \cite{Z} and references there) and degenerated formulas
for $\lambda=ir, \pi+ir$, we see readily that each character
of   $(\mathbb N_0,\tilde*)$ has a positive integral representation
w.r.t.~characters of  $(\mathbb N_0,*)$. As for the hypergroup
 $(\mathbb N_0,*)$, the support of the Plancherel measure is equal to 
$(\mathbb N_0,*)^\wedge$, we conclude that the generalized association scheme 
 $(\mathbb Z,\mathbb N_0, (R_i)_{k\in\mathbb N_0 }, (\tilde S_k)_{k\in \mathbb N_0 })$
has the positive connection property.
We thus may apply Theorem \ref{main-g}. The associated dual convolution can be
determined explicitely similar to \cite{Z}.

Notice that the hypergroups $(\mathbb N_0,*)$ and   $(\mathbb N_0,\tilde*)$
 are related by a deformation of the convolution via some positive character as described in \cite{V1}.
\end{example}

\bibliographystyle{amsplain}

\end{document}